\documentclass[11pt]{amsart}
\usepackage{amsmath,amsthm, amscd, amssymb, amsfonts}
\usepackage[all]{xy}
\usepackage{mathrsfs}
\usepackage{enumitem}

\allowdisplaybreaks
\usepackage[ansinew]{inputenc}
\usepackage{graphicx,fancyhdr}

\usepackage[dvips, dvipsnames, usenames]{color}

\newcommand{\comment}[1]{}
\newcommand{\qbinom}[3]{ \begin{pmatrix} #1\\#2\end{pmatrix}_{#3}}

\numberwithin{equation}{section}
\newtheorem{theorem}{Theorem}[section]
\newtheorem{lemma}[theorem]{Lemma}
\newtheorem{lem}[theorem]{Lemma}
\newtheorem{coro}[theorem]{Corollary}
\newtheorem{conjecture}[theorem]{Conjecture}
\newtheorem{prop}[theorem]{Proposition}
\newtheorem{pro}[theorem]{Proposition}

\theoremstyle{definition}

\newtheorem{example}[theorem]{Example}

\theoremstyle{remark}
\newtheorem{remark}[theorem]{Remark}

\def\pf{\begin{proof}}
\def\epf{\end{proof}}

\newcommand{\ba}{ \mathbf{a}}

\newcommand{\ku}{ \Bbbk}

\newcommand{\G}{\mathbb G}

\newcommand{\qmb}{\mathtt{q}}
\newcommand{\qt}{\widetilde{q}}

\newcommand{\wt}{\widetilde{w}}

\newcommand{\I}{\mathbb I}

\newcommand{\J}{\mathbb J}
\newcommand{\N}{\mathbb N}
\newcommand{\Q}{\mathbb Q}

\newcommand{\bq}{\mathbf{q}}

\newcommand{\Z}{\mathbb Z}
\newcommand{\zt}{\Z^{\theta}}

\newcommand{\cB}{\mathcal{B}}

\newcommand{\cJ}{\mathcal{J}}

\newcommand{\cR}{\mathcal{R}}

\newcommand{\X}{\mathcal{X}}

\newcommand{\ad}{\operatorname{ad}}

\newcommand{\Aut}{\operatorname{Aut}}

\newcommand{\id}{\operatorname{id}}

\newcommand{\GK}{\operatorname{GKdim}}

\newcommand{\ord}{\operatorname{ord}}

\def\ydh{{}^{H}_{H}\mathcal{YD}}

\newcommand{\NA}{\mathcal{B}}
\newcommand{\toba}{\mathcal{B}}
\newcommand{\ot}{\otimes}
\newcommand{\realroots}{\boldsymbol{\Delta }^{\mathrm{re}}}
\newcommand{\roots}{\boldsymbol{\Delta }}

\newcommand{\ydG}{{}^{\ku \Gamma }_{\ku \Gamma }\mathcal{YD}}

\newcounter{tabla}\stepcounter{tabla}

\begin{document}

\title[On finite GK-dimensional Nichols algebras of diagonal type]{On finite GK-dimensional Nichols algebras of diagonal type}  

\author[Andruskiewitsch; Angiono; Heckenberger]
{Nicol\'as Andruskiewitsch, Iv\'an Angiono, Istv\'an Heckenberger}

\address{FaMAF-CIEM (CONICET), Universidad Nacional de C\'ordoba,
Medina A\-llen\-de s/n, Ciudad Universitaria, 5000 C\' ordoba, Rep\'
ublica Argentina.} \email{(andrus|angiono)@famaf.unc.edu.ar}

\address{Philipps-Universität Marburg,
	Fachbereich Mathematik und Informatik,
	Hans-Meerwein-Straße,
	D-35032 Marburg, Germany.} \email{heckenberger@mathematik.uni-marburg.de}

\thanks{The work of N. A. and I. A.  was partially supported by CONICET,
Secyt (UNC), the MathAmSud project GR2HOPF. The work of I. A. was partially supported by ANPCyT (Foncyt).
The work of N. A., respectively I. A., was partially done during a visit to the University of Marburg, 
respectively the MPI (Bonn), both visits supported by the Alexander von Humboldt Foundation.}

\begin{abstract}
It was conjectured in \texttt{\small arXiv:1606.02521} that a Nichols algebra of diagonal type 
with finite Gelfand-Kirillov dimension has finite (generalized) root system.
We prove the conjecture assuming that the rank is 2.
We also show that a Nichols algebra of affine Cartan type has infinite Gelfand-Kirillov dimension. 
\end{abstract}

\maketitle


\section{Introduction}\label{section:introduction}

In this paper we contribute to the classification of Hopf algebras with finite
Gelfand-Kirillov dimension, $\GK $ for short. Specifically we propose:

\begin{conjecture}\label{conjecture:nichols-diagonal-finite-gkd} \cite[Conjecture 1.5]{AAH}  If $V$ is a 
 braided vector space of diagonal type and dimension $\theta \in \N$ such that the $\GK$ of its Nichols algebra
$\NA(V)$ is finite, then its (generalized) root system is finite.
\end{conjecture}

Since the classification of  the Nichols algebras of diagonal type with finite generalized root system is known \cite{H-classif},
a positive answer to the Conjecture \ref{conjecture:nichols-diagonal-finite-gkd}  would imply 
the classification of the Nichols algebras of diagonal type with finite $\GK$.
Indeed, the converse in the Conjecture is clearly true.
The defining relations of these Nichols algebras are also known \cite{A-jems}.
The $\GK$ of these Nichols algebras can be computed \eqref{eq:gkd-formula}.
Our main result provides partial answers to this Conjecture:

\begin{theorem}\label{theorem:main}
Conjecture \ref{conjecture:nichols-diagonal-finite-gkd} holds in the following cases:
\begin{enumerate}[leftmargin=*,label=\rm{(\alph*)}]
\item\label{it:cartan-affine} $V$ is of affine Cartan type.
\item\label{it:theta-2}  $\theta = 2$.
\end{enumerate}

\end{theorem}

We collect some necessary definitions and concepts in Section \ref{section:preliminaries}; 
Section \ref{section:diagonal-type} is devoted to general results that might be of  interest elsewhere. 
Part \ref{it:cartan-affine} is proved in Proposition \ref{prop:dimaffineCartan}.
Part \ref{it:theta-2}, which is Theorem \ref{th:infGK}, is proved in \S \ref{subsec:proof-theta-2}, after some preparatory Lemmas in 
\S \ref{subsection:rank-2}.

\section{Preliminaries}\label{section:preliminaries}

\subsection{Conventions}\label{subsection:conventions}
Let $\ku$ be an algebraically closed field of characteristic zero. All the vector spaces, algebras and tensor products  are over $\ku$.
Let $\N  =\{1,2,\dots\}$, $\N_0 = \N \cup 0$. 
Given $\theta \in\N$, we set $\I_\theta=\{1,2,\dots,\theta\}$, or simply $\I$ if $\theta$ is clear from the context.
In the polynomial ring $\Z[\qmb]$, we denote 
\begin{align*}
(n)_\qmb &=\sum_{j=0}^{n-1}\qmb^{j}, & (n)_\qmb^!&=\prod_{j=1}^{n} (j)_\qmb,& n &\in \N_0.
\end{align*}
 If  $q\in\ku$, then $(n)_q$, $(n)_q^!$ are the  evaluations at $q$.

\medbreak We denote by $\widehat \Gamma$ the group of multiplicative characters (one-dimensional representations) of a group $\Gamma$.
Let $\G_N$ be the group of roots of unity of order $N$,  $\G_N'$ the subset of primitive roots of order $N$, and
$\G_{\infty} = \bigcup_{N\in \N} \G_N$.

\medbreak Let $H$ be a Hopf algebra (always with bijective antipode). 
A braided Hopf algebra means a Hopf algebra in the category $\ydh $
of Yetter-Drinfeld modules over $H$.
If $R$ is a Hopf algebra  in $\ydh$, then $R\# H$ is the bosonization of $R$ by $H$.
Let $\ad$ be the adjoint action of $R\# H$ and $\ad_c$ the braided adjoint  action of $R$.
Then $\ad_c x\otimes \id = \ad (x\# 1)$ for $x\in R$. If $x\in R$ is primitive, then
$\ad_c x (y) = xy - \text{multiplication} \circ c (x\otimes y)$ for all $y\in R$.

\medbreak If $(V, c)$ is a braided vector space, then 
$\NA(V) = T(V)/\cJ(V)$ is the Nichols algebra of $V$, see \cite{AS Pointed HA,A-leyva,AA-diag-survey} for surveys on this notion.

\medbreak
A \emph{braided vector space of diagonal type} is a pair  $(V, c)$, where $V$
is a vector space   of dimension $\theta$ with a basis $(x_i)_{i\in \I_\theta }$, and the braiding
$c \in GL(V \ot V)$ is given by $c(x_i\otimes x_j)=q_{ij}x_j\otimes x_i$ for all $i,j\in \I$;
here $\bq = (q_{ij})_{i,j\in \I_\theta }\in(\ku^{\times})^{\theta \times \theta }$, $q_{ii} \neq 1$.
Let 
\begin{align*}
\qt_{ij}&:= q_{ij}q_{ji},& i\neq j \in \I_\theta.
\end{align*}
The generalized Dynkin diagram of $\bq$ is a graph with set of points $\I_\theta$ with the vertex $i$ decorated with $q_{ii}$; 
and for  $i\neq j \in \I_\theta$, no edge between $i$ and $j$ when $\qt_{ij} = 1$, otherwise there is an edge decorated with $\qt_{ij}$.

\subsection{The Gelfand-Kirillov dimension}\label{subsection:basic-defs}
A comprehensive exposition is \cite{KL}. For further use, 
we recall statements from \cite[\S 2.3.2]{AAH} inspired by \cite[Lemma 19]{R quantum groups}.

\begin{lemma}\label{lemma:rosso-lemma19-gral} Let $\NA = \oplus_{n \ge 0} \NA^n$
be a finitely generated graded algebra with $\NA ^0 = \ku$.
Let $(y_k)_{k\ge 0}$ be a family of homogeneous elements of $\NA$ such that
\begin{align} \label{equation:S-LI}
(y_{i_1}\dots y_{i_l}: i_j \in\N, i_1 <  \dots < i_l )
\end{align}
is a family of linearly independent elements. If there exist $m, p\in \N$ such that
$\deg y_i \leq mi + p$, for all $i\in \N$, then $\GK\cB = \infty$. \qed
\end{lemma}

The following Lemma is due to Rosso \cite[Lemma 14,\,Corollary 18]{R quantum groups}.
Let $(U,c)$ be a braided vector space of diagonal type, 
with respect to a basis $x_1, x_2$ and a matrix  $(q_{ij})_{i,j\in \I_2}$.  We set
\begin{align}\label{eq:def-muk-yk}
\mu_k &= \prod _{i=0}^{k-1}(1-q_{11}^i\qt_{12}), & y_k &= (\ad_c x_1)^k x_2.
\end{align}

\begin{lemma}\label{lemmata:rosso} 
\begin{enumerate}[leftmargin=*,label=\rm{(\alph*)}]
\item\label{item:diagonal-relations}  If $k\in \N$, then
$y_k=0$ iff $\mu_k x_1^k =0$,  iff $(k)_{q_{11}}! \mu_k =0$.
\item\label{lemma:rosso-cor18}  If $y_k\neq 0$ for all $k\in \N$, then the  set
\eqref{equation:S-LI} is linearly independent. 

\item\label{item:diagonal-infiniteGK}  If $(k)_{q_{11}}! \mu_k \neq 0$ for every $k\in \N$, then
$\GK \NA (U) = \infty$.

\item \label{lemma:points-trivial-braiding} If  
$q_{11} = 1$ and $\qt_{12} \neq 1$, then $\GK \NA(U) = \infty$. \qed
\end{enumerate}
\end{lemma}

\subsection{Nichols algebras of diagonal type}\label{subsection:diagonal-type}

We fix a braided vector space of diagonal type  $(V, c)$ with notation as in \S \ref{subsection:conventions}. 
\emph{We assume that the generalized Dynkin diagram is connected}.

\medbreak
Let $(\alpha_i)_{i\in \I}$ be the canonical basis of $\Z^\theta$.
Let $\bq$ be the bicharacter  on $\Z^\theta$ such that
 $\bq (\alpha _i,\alpha _j) = q_{ij}$ for all $i,j\in \I $; we set $q_{\alpha \beta }=\bq (\alpha ,\beta)$
 for $\alpha, \beta \in \Z^\theta $.

 Then $T(V)$ and $\NA (V)$ are $\Z^\theta $-graded with $\deg x_i=\alpha _i$ for all $i\in \I$. For each $\alpha \in \Z^\theta$, $\NA^{\alpha}(V)$ denotes the homogeneous component of degree $\alpha$.

\medbreak
Let $\Gamma = \zt$. We realize $V$ in $\ydG$ by choosing the family  $(\chi_i)_{i\in \I}$ in $\widehat{\Gamma}$ 
such that $\chi_j(\alpha_i) = q_{ij}$ for all $i,j\in \I$.
Then $\NA (V)$ becomes an $\zt$-graded object in $\ydG$.
There are skew-derivations $\partial _i$, $i\in \I$ of $\NA
(V)$, such that $\partial _i(x_j)=\delta_{ij}$ and
\begin{align*}
\partial_i(xy)&=x\partial_i(y)+\partial_i(x)(\alpha_i\cdot y),&  x,y &\in \NA (V).
\end{align*}

Let $\mu_k$, $y_k = (\ad_c x_1)^k x_2$ as in \eqref{eq:def-muk-yk}.  We notice that
\begin{align}\label{eq:derivations-yk}
\partial_1(y_k)&=0, & \partial_2(y_k)&=\mu_k x_1^k, & \mbox{for all }&k\in\N_0.
\end{align}
Also, it is well-known, and easy to check by a recursive argument, that
\begin{align} \label{eq:copr_um}
\Delta (y_k)=y_k\ot 1 +\sum _{i=0}^k\qbinom k i{q_{11}} \frac{\mu_k}{\mu_i}
x_1^{k-i}\ot y_i.
\end{align}

By \cite{Kh}, there is a totally ordered subset $L \subset \NA(V)$ consisting of $\Z^\theta$-homoge\-neous elements such that
\begin{align}\label{eq:PBW-Kharchenko}
\{ \ell_1^{m_1}\cdots \ell_k^{m_k} \,|\,k\in \N _0,\ell_1>\cdots >\ell_k\in L,0 < m_i< N_{\ell_i}\,\text{for all $i \in \I_k$}\}
\end{align}
is a linear basis of $\NA(V)$ (a so called restricted  PBW basis); here
$$ N_\ell = \min\{n\in \N: (n)_{q_{\deg \ell,\deg \ell}}=0\} \in \N \cup \infty$$
is called the height of $\ell$.

\begin{lem}\label{lem:infinite-generators-pbw}
If $\toba(V)$ has a restricted homogeneous PBW basis with infinitely many PBW generators of infinite height, then $\GK \toba(V) = \infty$.
\end{lem}

\pf By assumption there exists $L$ as above and $I \subseteq L$ infinite such that $N_\ell = \infty$ for all $\ell \in I$.
Let $d \in \N$, $F_d \subset I$ with $\vert F_d \vert = d$ and  $V_d$ the subspace generated by $1$,
the $x_i$'s and  all $\ell \in F_d$.
Then $(V_d)^{n + 1}$ contains the ordered monomials in $F_d$ of degree $\leq n + 1$, hence $\dim (V_d)^{n+1} \geq \binom{n+d}{d}$. 
Hence $\limsup \log_n \dim (V_d)^{n + 1} \geq d$. Since $d$ is arbitrary, $\GK \toba(V) = \infty$.
\epf

Assume that $L$ is finite (the Conjecture \ref{conjecture:nichols-diagonal-finite-gkd} says that this is the case when $\GK \NA(V) < \infty$). 
By \cite[Theorem 12.6.2]{KL}, we conclude that 
\begin{align} \label{eq:gkd-formula}
\GK \NA(V) = \vert \{\ell \in L: N_\ell =\infty \}\vert.
\end{align}

Let $\roots^V_+ =\roots_+ = (\deg \ell)_{\ell\in L}$ 
be the family of positive roots of $\NA (V)$ (with multiplicities).
By \cite[Lemma~4.7]{HS-RsWg},
it is uniquely determined, i.e. it does not depend on $L$.

\medbreak
 We say that \emph{we can reflect $V$ at} $i\in\I$  if, for all $j\neq i$, there exists $n\in\N_0$ such that
 $(n+1)_{q_{ii}}(1-q_{ii}^n q_{ij}q_{ji})=0$. 
  In such case, following \cite{H-inv} we define a generalized Cartan matrix $(c_{ij})$ by  $c_{ii} =2$ and
 \begin{align}\label{eq:defcij}
 c_{ij}&:=-\min\{n\in\N_0:(n+1)_{q_{ii}}(1-q_{ii}^nq_{ij}q_{ji})=0\},& & &j\neq i.
 \end{align}

\bigbreak
Let $s_i\in GL(\zt)$ be given by 
\begin{align}\label{eq:def-si}
s_i (\alpha_j) &= \alpha_j - c_{ij}\alpha_i, &  j & \in \I.
\end{align}
The reflection at the vertex $i$ of $\bq$ is the matrix $\cR^i(\bq) = (t_{jk})_{j,k\in\I}$, where
\begin{align}\label{eq:def-rho-ij}
t_{jk}&:= q_{s_i(\alpha_j), s_i(\alpha_k)} 
=  q_{jk}q_{ik}^{-c_{ij}}q_{ji}^{-c_{ik}}q_{ii}^{c_{ij} c_{ik}}, & j,k&\in\I.
\end{align}
Let $\cR^i(V)$ be the braided vector space of diagonal type with matrix $\cR^i(\bq)$.

 \begin{theorem}\label{th:GK-reflections} \cite{H-inv, AA}
  $\GK \toba(\cR^i(V)) = \GK \toba(V)$. \qed
 \end{theorem}

We say that $V$ \emph{admits all reflections} if we can reflect $V$ at every $i_1\in \I$,
then we can reflect $\cR^{i_1}(V)$ at every $i_2\in \I$ and so on, we can reflect $\cR^{i_k} \dots \cR^{i_1}(V)$
at every $i_{k+1}\in \I$ for all $k$.

If  $V$ admits all reflections, then we denote by $\X$ the collection of all braided vector spaces of diagonal type  obtained
from $V$ by a finite number of successive reflections at various vertices.
Here any two braided vector spaces with the same braiding matrix are identified.
The collection $(\roots^{U}_+)_{U \in \X}$ is the \emph{generalized root system} of $V$.

 \begin{remark}\label{rem:all-reflections-gkd}
If $\GK \NA(V) < \infty$, then we can reflect $V$ at every $i\in\I$ by Lemma \ref{lemmata:rosso} \ref{item:diagonal-infiniteGK}. 
Hence 	$V$ admits all reflections by Theorem \ref{th:GK-reflections}.
\end{remark}

\section{General results}\label{section:diagonal-type}
Recall that $(V,c)$ is of \emph{Cartan type} if there exist $a_{ij} \in \Z_{\leq 0}$ such that
\begin{align}\label{eq:cartan}
q_{ij}q_{ji} &= q_{ii}^{a_{ij}}, & i\neq j &\in \I.
\end{align}
Set $a_{ii} = 2$, $i\in \I$. 
If $q_{ii} \in \G_{\infty}$, then we choose $a_{ij} \in (-\ord q_{ii} , 0]$, when $i\neq j$; otherwise it is uniquely determined.
In any case, $\ba = (a_{ij})_{i,j\in \I_\theta}$ is an indecomposable 
symmetrizable generalized Cartan matrix \cite{K}.
These matrices are of three types: finite, affine or indefinite.
If $(V,c)$ is of Cartan type and $\GK \NA(V) < \infty$, then Conjecture \ref{conjecture:nichols-diagonal-finite-gkd} predicts that $\ba$ is of finite type.
Here is the confirmation for the affine type.

\begin{prop} \label{prop:dimaffineCartan}
If $\ba$ is of affine type, then $\GK\cB(V)=\infty$.
\end{prop}

\begin{proof}
Let $\realroots $ denote the set of real roots corresponding to $\ba$.
There exists a positive imaginary root $\delta $ such that
$\realroots +\delta=\realroots $ \cite[Proposition 6.3 d)]{K}. Let $m$ be the height of $\delta $ and let
$\alpha $ be a simple root. Choose a homogeneous restricted PBW basis of $\NA(V)$.
Then for all $k\ge 0$ there exists a PBW generator $y_k$ of degree
$\alpha+k\delta $, hence $\deg y_k = mk + 1$.
Therefore $\GK\cB(V)=\infty$ by Lemma \ref{lemma:rosso-lemma19-gral}.
\end{proof}

An indecomposable generalized Cartan matrix is \emph{compactly hyperbolic} if it is of indefinite type
and every proper minor is of finite type. If $\ba\in\Z^{\theta\times\theta}$ is compactly hyperbolic, then $\theta \leq 5$. In fact, the classification of compactly hyperbolic generalized Cartan matrices is known \cite{carbone}; there are
the matrices $\begin{pmatrix} 2 & a \\ b & 2\end{pmatrix}$ with $ab > 4$;
31 matrices in $\Z^{3 \times 3}$;
3 matrices in $\Z^{4 \times 4}$;
1 matrix in $\Z^{5 \times 5}$.
To prove Conjecture \ref{conjecture:nichols-diagonal-finite-gkd} in the Cartan case, 
it would be enough to verify it for compactly hyperbolic generalized Cartan matrices with $3 \leq \theta \leq 5$, as the case $\theta = 2$
is taken care by Theorem \ref{th:infGK}.

\medbreak Back to the general diagonal type,
we distinguish three classes of Nichols algebras. Given $\bq$ as above, we say that

\begin{itemize} [leftmargin=*]\renewcommand{\labelitemi}{$\circ$}
\item $\bq$ is of \emph{torsion class} if $q_{ii}, q_{ij}q_{ji} \in \G_{\infty}$ for all $i\neq j \in \I$;

\medbreak
\item $\bq$ is \emph{generic}, if $q_{ii} \notin \G_{\infty}$, and $q_{ij}q_{ji}=1$ or $q_{ij}q_{ji}\notin \G_{\infty}$,
for all $i\ne j\in \I$.

\medbreak
\item $\bq$ is \emph{semigeneric}
if it is neither generic nor of torsion class.
\end{itemize}

\begin{remark}\label{rem:nichols-diagonal-finite-gkd} 

\

\begin{enumerate} [leftmargin=*]
\item If $\bq$ is of torsion class, then Conjecture \ref{conjecture:nichols-diagonal-finite-gkd} says that $\GK \NA(V) < \infty$
implies $\GK \NA(V)  = 0$. Indeed all roots are real by \cite{CH}, and they would have finite non-trivial 
order by Lemma \ref{lemmata:rosso} \ref{lemma:points-trivial-braiding},
hence \eqref{eq:gkd-formula} applies.

\medbreak
\item\label{item:torsion-X-finite}
If $\bq$ is of torsion class, then the set $\X$ defined after Theorem \ref{th:GK-reflections} is finite. Indeed,
there are finitely many matrices with the shape \eqref{eq:def-rho-ij}.

\medbreak
\item\label{item:rosso-aa} \cite{R quantum groups, AA} If $\bq$ is generic, then $\GK \NA(V) < \infty$ if and only if 
there exists a Cartan matrix of finite type $\ba = (a_{ij})$, with symmetrizing diagonal matrix $(d_i)$,
and $q\notin\G_{\infty}$ such that $q_{ii} = q^{2d_i}$ and $q_{ij}q_{ji} = q^{2d_ia_{ij}}$ for all $i\neq j\in \I$.
Thus Conjecture \ref{conjecture:nichols-diagonal-finite-gkd} holds in this case.

\medbreak
\item A semigeneric matrix with finite generalized root system is either of super type or else one of 
two exceptions of ranks 2 and 4:
\begin{align*}
\xymatrix{\overset{q} {\circ} \ar  @{-}[r]^{q^{-1}}  & \overset{\omega}{\circ},}& &
\xymatrix{\overset{q} {\circ} \ar  @{-}[r]^{q^{-1}}  & \overset{q}{\circ} \ar  @{-}[r]^{q^{-1}} &
\overset{-1} {\circ} \ar  @{-}[r]^{-q}  & \overset{-q^{-1}}{\circ}.}
\end{align*}
Here $\omega\in \G'_3$ and $q\notin \G_{\infty}$; the first corresponds to Yamane's exotic quantum groups \cite{Y}
while the second is the row 14 in \cite[Table 3]{H-classif}.
\end{enumerate}
\end{remark}

\subsection{Semigeneric diagonal type}\label{subsection:generic}

\medbreak 
Let us fix $\bq =(q_{ij})_{i, j\in \I}$  semigeneric with $\GK\NA(V) < \infty$.
Let $\J = \{i\in \I: q_{ii} \notin \G_{\infty}\}$ be the set of generic points of $\bq$ and let
$\J_1, \dots \J_t$ be the connected components of the generalized Dynkin diagram $\bq =(q_{ij})_{i, j\in \J}$.

\begin{lemma}\label{lemma:semigeneric}
  If $i\in \I$ and $j\in \J$, then there exists $h\in \N_0$ such that $q_{jj}^{-h} = q_{ij}q_{ji}$.
  In particular,  either $q_{ij}q_{ji} = 1$ or $\notin \G_{\infty}$.
  \end{lemma}
\pf By Lemma \ref{lemmata:rosso}.
\epf

\begin{lemma}\label{lemma:semigeneric2}
If $i\notin \J$, $j\in \J$ and $q_{ij}q_{ji} \neq 1$, then either $\ord q_{ii} = 2$ and $q_{ij}q_{ji} = q_{jj}^{-h}$ with $h\in \I_2$;
or else $\ord q_{ii} = 3$ and $q_{ij}q_{ji} = q_{jj}^{-1}$.
  \end{lemma}
\pf First, there exists $h\in \N$ such that $q_{jj}^{-h} = q_{ij}q_{ji}$ by Lemma \ref{lemma:semigeneric}.
Let $N = \ord q_{ii}$.
We apply the reflection at $i$:
\begin{align*}
 \xymatrix{\overset{q_{ii}} {\circ} \ar  @{-}[r]^{q_{jj}^{-h}}  & \overset{q_{jj}}{\circ}&  \ar@/^1pc/@{<->}[r]^{i} &  &
\overset{q_{ii}} {\circ} \ar  @{-}[r]_{q_{ii}^2q_{jj}^{h} \quad}  & \overset{q_{ii}q_{jj}^{1-h(N - 1)}}{\circ}.}
\end{align*}
Then either $1 = h(N - 1)$ that gives $h= 1$, $N = 2$; or else
there exists $t\in \N$ such that $(q_{ii}q_{jj}^{1-h(N - 1)})^{-t} = q_{ii}^2q_{jj}^{h}$ by Lemma \ref{lemma:semigeneric}.
A straightforward analysis yields the claim.
\epf

As a consequence we derive the corresponding version of Theorem \ref{th:infGK} for semigeneric braidings. It will be useful for the proof of the general case.

\begin{coro}\label{coro:infGK-semigeneric}
Let $V$ be a braided vector space of \emph{semigeneric} diagonal type and dimension $2$ such that the $\GK$ of its Nichols algebra
$\NA(V)$ is finite. Then its generalized root system is finite.
\end{coro}
\pf
We may assume that $q_{11} \in \G_{\infty}$, $q_{22} \notin \G_{\infty}$ up to reflection. Indeed, if neither $q_{11}$ nor $q_{22}$ 
belong to $\G_{\infty}$, then $V$ is generic by Lemma \ref{lemmata:rosso} \ref{item:diagonal-infiniteGK}.
So, either $q_{11} \in \G_{\infty}$ or else $q_{22} \in \G_{\infty}$. 
If both belong to $\G_{\infty}$, then $q_{ij}q_{ji}\notin \G_{\infty}$. 
Applying reflection at $1$, we have that the new $q_{22}\notin \G_{\infty}$. 

By Lemma \ref{lemma:semigeneric2}, the Dynkin diagram of $V$ is one of the following:
\begin{align*}
& \xymatrix{\overset{-1} {\circ} \ar  @{-}[r]^{q^{-1}}  & \overset{q}{\circ}}, &
& \xymatrix{\overset{-1} {\circ} \ar  @{-}[r]^{q^{-2}}  & \overset{q}{\circ}}, &
& \xymatrix{\overset{\zeta} {\circ} \ar  @{-}[r]^{q^{-1}}  & \overset{q}{\circ}}, &
& q\notin \G_{\infty}, \, \zeta\in\G_3'.
\end{align*}
All of them appear in \cite[Table 1]{H-classif}, so $V$ has a finite root system.
\epf

\subsection{Nichols algebras of indefinite Cartan type}\label{subsection:affine-indefinite}

Let $A$ be an indecomposable generalized Cartan matrix. 
Let $W$ be the corresponding Weyl group, see \cite{K}.

\begin{lemma} \label{le:indefinite}
Assume that $A$ is of indefinite type.
Let $Q$ be the root lattice
corresponding to $A$ and let $Q_+\subset Q$ be the submonoid generated by the
simple roots. Then $W\gamma \cap Q_+$ is infinite for all $\gamma \in Q_+- 0$.
\end{lemma}

\begin{proof}
Let $\pi $ denote the set of simple roots.
Let $\gamma \in Q_+- 0$. Assume that $W\gamma \cap Q_+$ is finite.
Let $\beta =\sum _{\alpha \in \pi }c_\alpha \alpha \in W\alpha \cap Q_+$,
where $c_\alpha \ge 0$ for all $\alpha \in \pi $, such that
$w\gamma -\beta \notin Q_+- 0$ for all $w\in W$.
Then for all $\alpha \in \pi $ there exists $m_\alpha\ge 0$ such that
$s_\alpha \beta =\beta -m_\alpha \alpha $.
On the other hand, $s_\alpha \beta =\beta -\sum _{\alpha '\in \pi }c_{\alpha
'}a_{\alpha \alpha '}\alpha $, that is, $A(c_{\alpha '})_{\alpha '\in \pi }\ge
0$. Since $A$ is of indefinite type
and $(c_{\alpha '})_{\alpha '\in \pi }$, $A(c_{\alpha '})_{\alpha '\in \pi }$ have only non-negative entries, we have a contradiction.
\end{proof}

Let  $V$ be a braided vector
space of Cartan type with Cartan matrix $A$.

\begin{lemma} \label{le:dimindefCartan}
  If there exists a root $\gamma \in \roots^V _+$ of $\cB (V)$
such that $q_{\gamma,\gamma}=1$, then $\GK\cB(V)=\infty$.
\end{lemma}

\begin{proof}
The Cartan matrix $A$ is not of finite type,
since otherwise $q_{ii}\ne 1$ for all $i$ and hence $q_{\gamma,\gamma}\ne 1$ for all roots $\gamma$.
If $A$ is of affine type, then $\GK \NA(V)=\infty$ by Proposition
\ref{prop:dimaffineCartan}. 
We assume then that $A$ is of indefinite type.
By \cite{H-inv}, $s_\alpha (\roots^V_+- \{\alpha \})
=\roots^V_+- \{\alpha \}$ for all simple roots $\alpha $.
Since $q_{w\gamma, w\gamma}=q_{\gamma,\gamma}=1$ for all $w\in W$,
all root vectors of degree $w\gamma $ with $w\in W$ have infinite height.
By Lemma~\ref{le:indefinite}, $W\gamma $ is infinite.
Thus $\cB (V)$ has a restricted homogeneous PBW
basis containing infinitely many PBW generators having infinite height. The claim follows by Lemma \ref{lem:infinite-generators-pbw}.
\end{proof}

\subsection{Braided coideal subalgebras}
Just in this Subsection, the field $\Bbbk$ is arbitrary. 
Let $H$ be a Hopf algebra with bijective antipode.

\begin{pro} \label{pr:K/I} \cite[Prop.\,2.1]{GH}.
Let $B$ be a bialgebra in $\ydh$, let $K$ be a subalgebra of $B$ and let  
$I$ be a subobject of $K$ in $\ydh$, such that it is a coideal of $B$, an ideal of $K$ and
\begin{align}\label{eq:coideal-subalgebra}
\Delta (K)\subseteq K\otimes K+I\otimes B.
\end{align}
Then $K/I$ inherits a structure of bialgebra  in $\ydh$ from $B$. 
\end{pro}

By \eqref{eq:coideal-subalgebra}, $K$ is a right coideal subalgebra of $B$.

\begin{proof} The existence of $\Delta$ is verified by usual chasing in the following commutative diagram:
\begin{align*}
\xymatrix@R-12pt{
B \ar@{->}[rrrr]^{\Delta} \ar@{->>}[dddd]^{\pi} \ &&&& B \otimes B  \ar@{->>}[dddd]^{\pi \otimes \pi}
\\
& K \ar@{_(->}[lu] \ar@{->}[rr]^{\Delta} \ar@{->>}[dd]^{\pi} && K\otimes K+I\otimes B \ar@{^(->}[ru] \ar@{->>}[dd]^{\pi \otimes \pi} & 
\\ \\
& K/I \ar@{_(->}[ld] \ar@{-->}[rr]^{\Delta} && K/I\otimes K/I \ar@{^(->}[rd] & 
\\
B/I \ar@{->}[rrrr]^{\Delta} \ &&&& B/I \otimes B/I}
\end{align*}
The associativity and compatibilities follow at once from those of $B$.
\end{proof}

We now apply Proposition~\ref{pr:K/I} in the following context. Recall that $\alpha_1,\alpha_2$ is the canonical basis of $\Z ^2$.
Let
$V=V_1\oplus V_2$
be a direct sum in $\ydh$. Then   $\NA (V)$  has a unique
$\N _0^2$-grading (as a Hopf algebra in $\ydh$) 
\begin{align*}
\NA (V) = \bigoplus_{\alpha \in  \N _0^2} \NA^{\alpha} (V)
\end{align*}
such that $\deg V_1=\alpha_1$ and $\deg V_2=\alpha_2$. Let $r\in \Q_{\ge 0}$. We set
\begin{align*}
B_{\geq r} &= \bigoplus_{\substack{\alpha = a_1\alpha_1+a_2\alpha_2 \in  \N _0^2: \\ a_1\ge ra_2}} \NA^{\alpha} (V), &
B_{> r} &= \bigoplus_{\substack{\alpha = a_1\alpha_1+a_2\alpha_2 \in  \N _0^2: \\ a_1 > ra_2}} \NA^{\alpha} (V),
\\
K_{\geq r}&=\{x\in \NA (V)\,|\, \Delta (x)\in B_{\geq r}\ot \NA (V)\},&
K_{>r}&=K_{\geq r}\cap B_{>r}.
\end{align*}

\begin{pro}
Let $r\in \Q $ with $r\ge 0$. Then the braided
bialgebra structure of $\NA (V)$
induces a braided Hopf algebra structure on $K_{\geq r}/K_{>r}$.
\end{pro}

\begin{proof} 
We claim that:
\begin{enumerate}[leftmargin=*,label=\rm{(\roman*)}]
\item\label{item:GH-appl1} $K_{\geq r}\subseteq B_{\geq r}$,

\item\label{item:GH-appl2} $K_{\geq r}$ is a subalgebra of $\NA (V)$ in $\ydh$,

\item\label{item:GH-appl3} $\Delta (K_{\geq r})\subseteq
K_{\geq r}\otimes K_{\geq r}+K_{>r}\otimes \NA (V)$,

\item\label{item:GH-appl4} $K_{>r}$ is an ideal of $K_{\geq r}$ and a coideal of $\NA (V)$ in $\ydh$.
\end{enumerate}

For \ref{item:GH-appl1}, apply $(\id \ot \varepsilon)$ to the  inclusion defining $K_{\geq r}$. Now \ref{item:GH-appl2} follows since the multiplication and the Yetter-Drinfeld structure of $\NA (V)$
are $\N _0^2$-graded. 

For \ref{item:GH-appl3}, note that $K_{\geq r}$ is a right coideal since $\Delta$ is coassociative. Using this fact and that $\Delta$ is $\N_0^2$-graded,
\begin{align*}
\Delta(K_{\ge r})&\subseteq K_{\ge r}\otimes B_{\ge r} +K_{>r}\otimes B(V).
\end{align*}
Indeed, let $x\in K_{\ge r}$ of degree $a_1\alpha_1+a_2\alpha_2$. 
We write $\Delta(x)=\sum_{i} y_i\otimes z_i$ with $y_i,z_i$ $\N_0^2$-homogeneous. If $y_i$, $z_i$ have degree $b_1\alpha_1+b_2\alpha_2$, $c_1\alpha_1+c_2\alpha_2$, then either $b_1>rb_2$ so $y_i\ot z_i \in K_{>r}\otimes B(V)$, or else $b_1=rb_2$, in which case $c_1=a_1-b_1\ge ra_2-rb_2=rc_2$ and $y_i\ot z_i \in K_{\ge r}\otimes B_{\ge r}$. Now
\begin{multline*}
(\Delta \ot \id ) \Delta (K_{\geq r}) \subseteq
(\Delta \ot \id )(K_{\geq r}\ot B_{\geq r} +K_{>r}\ot \NA (V))\\
 \subseteq K_{\geq r}\ot K_{\geq r} \ot B_{\geq r}
+ K_{\geq r}\ot K_{>r}\ot \NA (V)
+K_{>r}\ot \NA (V)\ot \NA (V).
\end{multline*}
We apply $(\id \ot \id \ot \varepsilon)$ to the previous inclusion and 
get \ref{item:GH-appl3}.

Finally we prove \ref{item:GH-appl4}: Let $x\in K_{>r}$, $y\in K_{\geq r}$, we may assume they are homogeneous of degrees $\alpha=a_1\alpha_1+a_2\alpha_2$, $\beta=b_1\alpha_1+b_2\alpha_2$, 
so $a_1>ra_2$, $b_1\geq rb_2$. Hence $xy\in K_{\geq r}$ by \ref{item:GH-appl2}, and $xy$ has degree $\alpha+\beta=(a_1+b_1)\alpha_1+(a_2+b_2)\alpha_2$ with $a_1+b_1 >r(a_2+b_2)$, so $xy\in K_{>r}$. Analogously, $yx\in K_{>r}$. Thus $K_{>r}$ is an ideal of $K_{\geq r}$. 
Now $\Delta(x)\in K_{\geq r}\otimes K_{\geq r}+K_{>r}\otimes \NA (V)$ by \ref{item:GH-appl3}.
For each $u\ot v \in K_{\geq r}\otimes K_{\geq r}$ appearing in $\Delta(x)$, $u$, $v$ homogeneous of degrees $c_1\alpha_1+c_2\alpha_2$, $d_1\alpha_1+d_2\alpha_2$, we have that $c_1\geq rc_2$, $d_1\geq rd_2$. As $\Delta$ is $\N_0^2$-graded, $c_1+d_1=a_1$, $c_2+d_2=a_2$, so either $c_1> rc_2$ or else $d_1> rd_2$; in other words, either $u\in K_{>r}$ or else $v\in K_{>r}$, so $\Delta(x)\in K_{\geq r}\otimes K_{>r}+K_{>r}\otimes \NA (V)$. Thus $K_{>r}$ is a coideal of $\NA (V)$ in $\ydh$.

\medbreak

Hence we may apply Proposition~\ref{pr:K/I} and $K_{\geq r}/K_{>r}$ is a 
bialgebra in $\ydh$. Since $\NA (V)$ is $\N _0$-graded and connected and $\Delta $ is $\N_0$-graded, $K_{\geq r}$ and $K_{>r}$ are also $\N _0$-graded. Therefore $K_{\geq r}/K_{>r}$ is $\N_0$-graded and connected, so it is a Hopf algebra in $\ydh$ by \cite[5.2.10]{Mo}.
\end{proof}

\section{Rank 2}

This Section contains a proof of the following result:

\begin{theorem}\label{th:infGK}
Let $V$ be a braided vector space of diagonal type and dimension $2$ such that  $\GK \NA(V) < \infty$. Then its generalized root system is finite.
\end{theorem}

Let $Z$ be a vector space and $(z_i)_{i\in I}$ a family of vectors in $Z$. Then $\langle z_i: i\in I \rangle$ denotes the subspace of $Z$ generated by the $z_i$'s.

\subsection{Lemmas for $\theta = 2$}\label{subsection:rank-2}

We establish some properties needed in the proof of Theorem~\ref{th:infGK}.
We keep the notation in \S \ref{subsection:diagonal-type};
we assume that $\theta = 2$ and set as above $y_k = (\ad x_1)^k x_2 \in \NA (V)$.
We also set:
\begin{align}\label{eq:def-pn}
\beta _m &=m\alpha_1+\alpha_2, &
p_m&=q_{\beta_{m} \, \beta_{m}} =q_{11}^{m^2}q_{12}^mq_{21}^mq_{22}, &
m&\in \N_0.
\end{align}

\subsubsection{On the powers of the $y_n$'s}\label{subsubsec:powers-yn}

\begin{lemma}\label{lem:derivation-auxiliar}
Let $r,s,t\in\N_0$, $r\le s$. Then
\begin{align*}
\partial _1^{t}\partial _2(y_ry_s) &=
\begin{cases}
0, & s<t; \\
\mu_s(s)_{q_{11}} (s-1)_{q_{11}} \dots (s-t+1)_{q_{11}} \, y_rx_1^{s-t}, & r<t\le s.
\end{cases}
\end{align*}
\end{lemma}
\pf
By \eqref{eq:derivations-yk} we have that
\begin{align}\label{eq:deriv-ysyr}
\partial_2(y_ry_s)&=\mu _r q_{21}^sq_{22} x_1^ry_s+\mu _sy_rx_1^s &
\text{for all }&r,s\in \N_0.
\end{align} 
Using \eqref{eq:derivations-yk} again, if $t>s$, then $\partial _1^{t}\partial _2(y_ry_s)=0$, and if 
$r<t\le s$,
\begin{align*}
\partial _1^{t}\partial _2(y_ry_s) &=\partial _1^{t}(\mu _sy_rx_1^s) 
=\mu_s(s)_{q_{11}} (s-1)_{q_{11}} \dots (s-t+1)_{q_{11}} \, y_rx_1^{s-t}
\end{align*}
by induction on $t$.
\epf

\begin{lemma} \label{le:yy1}
Let $l\in \N _0$. Assume that
\begin{align*}
x_1^{2l+1}&\ne 0,& x_2^2&\ne 0,& y_{l+1}&\ne 0,&
y_{l+1}y_l &\in \langle  y_ry_s: \,0\le r\le s\rangle.
\end{align*}
Then  $q_{11}^{l(l+1)}(\qt_{12})^{l+1}q_{22}=1$ and
\begin{align} \label{eq:yy}
y_{l+1}y_l=q_{11}^{l(l+1)}q_{12}^{l+1}q_{21}^lq_{22}y_ly_{l+1}.
\end{align}
\end{lemma}

\begin{proof} The last assumption of the lemma implies that
$$y_{l+1}y_l\in \langle  y_ry_s: \,0\le r\le s, \,\, r+s=2l+1 \rangle.$$
Using Lemma \ref{lem:derivation-auxiliar} we have that
\begin{align*}
\partial _1^{l+2}\partial _2(y_{l+1}y_l)&=0, &
\partial _1^{l+2}\partial _2(y_ry_s) &\in \Bbbk y_rx_1^{s-l-2}-0,
\end{align*}
for all $r,s\in \N_0$ with $r+s=2l+1$, $s\ge l+2$ such that $y_s\neq 0$. Thus there exists
$\lambda \in \ku $ such that
$y_{l+1}y_l=\lambda y_ly_{l+1}$. Now we apply $\partial _1^{l+1}\partial _2$ to this
equation. Using that
$\partial _1^{l+1}(x_1^l)=0$, $\partial _1^{l+1}(x_1^{l+1})=(l+1)_{q_{11}}^!\ne 0$ we obtain
$$ \mu _{l+1}q_{21}^lq_{22}(q_{11}^lq_{12})^{l+1}y_l=\lambda \mu
_{l+1}y_l.$$
Thus $\lambda =q_{11}^{l(l+1)}q_{12}^{l+1}q_{21}^lq_{22}$, since $y_{l+1}\ne 0$.
Finally we apply $\partial _2^2$ to \eqref{eq:yy}. Since $\mu _{l+1}\ne 0$ and $x_1^{2l+1}\neq 0$, we obtain that
\begin{align*}
q_{21}^lq_{22}+q_{21}^l=q_{11}^{l(l+1)}q_{12}^{l+1}q_{21}^lq_{22}(q_{21}^{l+1}q_{22}+q_{21}^{l+1}).
\end{align*}
As $x_2^2\ne 0$, we have that $q_{22}\ne -1$.
Hence $q_{11}^{l(l+1)}(\qt_{12})^{l+1}q_{22}=1$.
\end{proof}

\begin{lemma} \label{le:yy2}
Let $l\in \N $. Assume that
$y_l^2\in \langle y_ry_s: \,0\le r<s \rangle$ and that
$y_l\ne 0$. Then $y_{l+1}=0$ and $q_{11}^{l^2}(\qt_{12})^lq_{22}=-1$.
\end{lemma}

\begin{proof}
As $\toba(V)$ is $\N_0$-graded, the last assumption says that
$$y_l^2\in \langle y_ry_s: \,0\le r<s,\, \, r+s=2l\rangle.$$
Using Lemma \ref{lem:derivation-auxiliar} we have that
\begin{align*}
\partial _1^{l+1}\partial _2(y_l^2)&=0, &
\partial _1^{l+1}\partial _2(y_ry_s) &\in \Bbbk y_rx_1^{s-l-1}-0,
\end{align*}
for all $r,s\in \N_0$ with $r<l$, $r+s=2l$, $y_s\ne 0$. Thus $y_l^2=0$.
Since $l>0$,
$$ 0=\partial _1^{l-1}\partial _2(y_l^2)= a_1x_1y_l+a_2y_lx_1 $$
for some $a_1,a_2\in \ku - 0$.
Since $x_1y_l=y_{l+1}+q_{11}^lq_{12}y_lx_1$ and since $y_{l+1}$ and $y_lx_1$
are linearly independent whenever $y_{l+1}=0$, the latter equation implies that
$y_{l+1}=0$. Therefore $x_1y_l=q_{11}^lq_{12}y_lx_1$ and we have that
$$0=\partial_2(y_l^2)=\mu _r(q_{21}^lq_{22}(q_{11}^lq_{12})^l+1)y_lx_1^l. $$
This implies the last claim.
\end{proof}


We fix $n>0$ and set $q=p_n$, cf. \eqref{eq:def-pn}. Assume that $q\in \mathbb{G}'_N$ 
for some $N\ge 2$.
In the next few Lemmas we prepare a condition for $y_n^N$ being a root vector.
We start with some computations with $q$-numbers. Recall that
\begin{align*}
(t_1+t_2)_q&=(t_1)_q+q^{t_1}(t_2)_q, & \mbox{for all }t_1,t_2 &\in \N _0. 
\end{align*}
In particular, $(N)_q=0$ and $(N-t)_q=-q^{-t}(t)_q$ for all $0\le t\le N$.

\begin{lem} \label{le:sumprod}
Let $r\in \{0,1,\dots ,N-2\}$ and $t\in \N _0$.
Then
$$\sum _{l=0}^t q^l(l+1)_q(l+2)_q\cdots
(l+r)_q=\frac{(t+1)_q(t+2)_q\cdots (t+r+1)_q}{(r+1)_q}.$$
\end{lem}

\begin{proof}
Note that $(r+1)_q\ne 0$ since $1\le r+1<N$.
We proceed by induction on $t$. For $t=0$ the claim is trivial.
For any $t\ge 0$, for which the claim holds, we obtain that
\begin{align*}
&\sum _{l=0}^{t+1}q^l(l+1)_q(l+2)_q\cdots (l+r)_q\\
&=\sum _{l=0}^tq^l(l+1)_q(l+2)_q\cdots (l+r)_q
+q^{t+1}(t+2)_q\cdots (t+r+1)_q\\
&=\frac{(t+1)_q(t+2)_q\cdots (t+r+1)_q}{(r+1)_q}+q^{t+1}(t+2)_q(t+3)_q\cdots (t+r+1)_q\\
&=\frac{(t+1)_q+q^{t+1}(r+1)_q}{(r+1)_q}(t+2)_q(t+3)_q\cdots (t+r+1)_q\\
&=\frac{(t+2)_q(t+3)_q\cdots (t+r+1)_q(t+r+2)_q}{(r+1)_q}.
\end{align*}
This proves the claim.
\end{proof}

For any $t\in \{0,1,\dots ,N-2\}$ let
\begin{align}\label{eq:Y(n)}
Y(t)&=\sum _{j=0}^t (q_{11}^nq_{12})^{-j} \frac{(N-t-1+j)_q^!}{(j)_q^!} y_n^{t-j}y_{n+1}y_n^j.
\end{align}
In particular, $Y(0)=(N-1)_q^! y_{n+1}$.

\begin{lem} \label{le:der_unN}
$\partial_1^{n-1}\partial_2(y_n^N)=- \mu_n (n)_{q_{11}}^!(q_{11}^nq_{12})^{-1}Y(N-2)$.
\end{lem}

\begin{proof}
First we obtain that
\begin{align*}
\partial_2(y_n^N)=&\sum_{l=0}^{N-1}y_n^{N-1-l}\partial_2(y_n)\alpha_2\cdot
y_n^l
=\mu_n \sum_{l=0}^{N-1} ( q_{21}^n q_{22} )^l
y_n^{N-1-l} x_1^n y_n^l.
\end{align*}
Since $\partial_1(y_n)=0$, we have that
\begin{align*}
\partial_1^{n-1}\partial_2(y_n^N)=&
\mu_n (n)_{q_{11}}^!
\sum_{l=0}^{N-1}(q_{11}^{n(n-1)}q_{12}^{n-1}q_{21}^nq_{22})^ly_n^{N-1-l}x_1y_n^l.
\end{align*}
The equation $y_{n+1}=x_1y_n-q_{11}^nq_{12}y_nx_1$ implies that
\begin{align} \label{eq:x1un}
x_1y_n^l=q_{11}^{nl}q_{12}^ly_n^lx_1
+\sum_{j=0}^{l-1}(q_{11}^nq_{12})^{l-1-j}y_n^{l-1-j}y_{n+1}y_n^j.
\end{align}
Therefore
\begin{align*}
\partial_1^{n-1}\partial_2(y_n^N)=&
\mu_n (n)_{q_{11}}^!
\sum_{l=0}^{N-1}(q_{11}^{n(n-1)}q_{12}^{n-1}q_{21}^nq_{22})^ly_n^{N-1-l}\cdot\\
&\Big(q_{11}^{nl}q_{12}^ly_n^lx_1+
\sum_{j=0}^{l-1}(q_{11}^nq_{12})^{l-1-j}y_n^{l-1-j}y_{n+1}y_n^j\Big)\\
=&\mu_n (n)_{q_{11}}^!
\Big(\sum_{l=0}^{N-1}q^ly_n^{N-1}x_1\\
&+(q_{11}^nq_{12})^{-1}
\sum_{j=0}^{N-2}\Big(\sum_{l=j+1}^{N-1}q^l\Big)(q_{11}^nq_{12})^{-j}
y_n^{N-2-j}y_{n+1}y_n^j\Big).
\end{align*}
Hence the Lemma follows since
$\sum_{l=0}^{N-1}q^l=0$ and $\sum_{l=j+1}^{N-1}q^l=-(j+1)_q$.
\end{proof}

For any $1\le t\le N-1$ let
\begin{align} \label{eq:dt}
d_t=1-q^{t+1}q_{11}^{2n}\qt_{12}
+\frac{q^t(1-q_{11}^n\qt_{12})(n+1)_{q_{11}}}{(t)_q}.
\end{align}
Observe that $d_t$ depends on $n$.

\begin{lem} \label{le:der_Ut}
Let $t\in \{1,2,\dots ,N-1\}$. Then
\begin{align*}
\partial_1^n\partial_2 (Y(t))=
\mu_n(n)_{q_{11}}^!(q_{11}^nq_{12})^{-1}d_tY(t-1).
\end{align*}
\end{lem}

\begin{proof}
Similarly to the calculation 
in Lemma~\ref{le:der_unN} we have that
\begin{align*}
\partial_1^n\partial_2(Y(t))=&\sum_{l=0}^t
(q_{11}^nq_{12})^{-l}\Big (\prod_{i=1}^{N-t-1}(l+i)_q\Big) \cdot\\
&\Big (\partial_1^n\partial_2(y_n)(t-l)_q q^{l+1}q_{11}^nq_{21}
y_n^{t-1-l}y_{n+1}y_n^l\\
&\phantom{\Big(}+ \partial_1^n\partial_2(y_n)(1-q_{11}^n\qt_{12})(n+1)_{q_{11}} q^l y_n^{t-l}
x_1 y_n^l\\
&\phantom{\Big(}+\partial_1^n\partial_2(y_n)(l)_q y_n^{t-l} y_{n+1}
y_n^{l-1}
\Big).
\end{align*}
Using \eqref{eq:x1un} this implies that
\begin{align*}
\partial_1^n&\partial_2 (Y(t))=
\partial_1^n\partial_2(y_n) \sum_{j=0}^{t-1}
(q_{11}^nq_{12})^{-j}\Big (\prod_{i=1}^{N-t-1}(j+i)_q\Big)
\cdot\\
&\quad (-q^{t-j})(N-t+j)_q q^{j+1}q_{11}^nq_{21} y_n^{t-1-j}y_{n+1}y_n^j\\
&+\partial_1^n\partial_2(y_n)(1-q_{11}^n\qt_{12})(n+1)_{q_{11}}
\sum_{l=0}^t (q_{11}^nq_{12})^{-l}
\Big (\prod_{i=1}^{N-t-1}(l+i)_q\Big)q^l\cdot \\
&\quad \Big(q_{11}^{nl}q_{12}^ly_n^t x_1
+\sum_{j=0}^{l-1}(q_{11}^nq_{12})^{l-1-j}y_n^{t-1-j}y_{n+1}y_n^j\Big)\\
&+\partial_1^n\partial_2(y_n)\sum_{j=0}^{t-1} (q_{11}^nq_{12})^{-j-1}
\Big(\prod_{i=1}^{N-t-1}(j+1+i)_q\Big)
(j+1)_q y_n^{t-1-j} y_{n+1} y_n^j.
\end{align*}
Lemma~\ref{le:sumprod} tells that
$$ \sum_{l=0}^tq^l\prod_{i=1}^{N-t-1}(l+i)_q=\prod_{i=1}^{N-t}(t+i)_q/(N-t)_q=0
$$
since $1<t<N$. Therefore the terms $y_n^tx_1$ disappear in the above expression
for $\partial_1^n\partial_2(Y(t))$.
Moreover,
\begin{align*}
\sum_{l=0}^t&\sum_{j=0}^{l-1}\Big(\prod_{i=1}^{N-t-1}(l+i)_q\Big)q^l
(q_{11}^nq_{12})^{-1-j}\\
=&\sum_{j=0}^{t-1}(q_{11}^nq_{12})^{-1-j}
\sum_{l=j+1}^tq^l\Big(\prod_{i=1}^{N-t-1}(l+i)_q\Big)\\
=&\sum_{j=0}^{t-1}(q_{11}^nq_{12})^{-1-j}
\Bigg(\sum_{l=0}^tq^l\Big(\prod_{i=1}^{N-t-1}(l+i)_q\Big)
-\sum_{l=0}^jq^l\Big(\prod_{i=1}^{N-t-1}(l+i)_q\Big)\Bigg)\\
=&\sum_{j=0}^{t-1}(q_{11}^nq_{12})^{-1-j}
\Big(\prod_{i=1}^{N-t}(t+i)_q-\prod_{i=1}^{N-t}(j+i)_q\Big)/(N-t)_q\\
=&-\sum_{j=0}^{t-1}(q_{11}^nq_{12})^{-1-j}
\Big(\prod_{i=1}^{N-t}(j+i)_q\Big)/(N-t)_q.
\end{align*}
Therefore,
\begin{align*}
\partial_1^n\partial_2(Y(t))=&
\partial_1^n\partial_2(y_n)(q_{11}^nq_{12})^{-1}
\sum_{j=0}^{t-1}(q_{11}^nq_{12})^{-j}
\Big(\prod_{i=1}^{N-t}(j+i)_q\Big)\cdot\\
&\Big(-q^{t+1}q_{11}^{2n}\qt_{12}
-\frac{(1-q_{11}^n\qt_{12})(n+1)_{q_{11}}}{(N-t)_q}+1\Big)y_n^{t-1-j}y_{n+1}y_n^j.
\end{align*}
Thus the claim follows from this equality and $(N-t)_q=-q^{-t}(t)_q$.
\end{proof}

\begin{pro}\label{prop:dt-yn}
Assume that $y_n^N=0$ and $y_{n+1}\ne 0$.
Then $d_t=0$ for some $1\le t\le N-2$, see \eqref{eq:dt}.
\end{pro}

\begin{proof}
Lemma~\ref{le:der_unN} implies that $Y(N-2)=0$.
Thus $$(\partial_1^n\partial_2)^{N-2}(Y(N-2))=0.$$
By Lemma \ref{le:der_Ut} either $d_{N-2}=0$ or else $Y(N-3)=0$. Recursively, if $Y(t)=0$, then either $d_{t}=0$ or else $Y(t-1)=0$, for each $1\le t\le N-2$.
Since $Y(0)=(N-1)_q^!y_{n+1}\ne 0$, necessarily $d_t=0$ for some $1\le t\le N-2$.
\end{proof}

\begin{lem}\label{le:yNneq0isroot}
If $y_n^N\ne 0$ then $N\beta _n$ is a root of $V$.
\end{lem}

\begin{proof}
By the definition of roots of $\NA (V)$, either
$\deg y_n^N=N(n\alpha_1+\alpha_2)$ is a root of $V$ or $y_n^N$
can be expressed as a linear combination of products $\ell_1^{m_1}\cdots \ell_k^{m_k}$ 
as in \eqref{eq:PBW-Kharchenko}, where each $\ell_i$ corresponds either to a Lyndon word $l_i$  greater than $x_1^nx_2$, or else to a power of this kind of letters. 

Assume the last case holds. Then each $l_i$ starts with $1^k2$,
where $k\le n$, and ends with $2$. Since $l_i$ is a Lyndon word, any end of $l_i$
is larger than $v$, and hence it contains no subword $1^l2$ with $l>k$. Therefore $l_i=x_1^{k_1}x_2\cdots x_1^{k_r}x_2$ with $k_1,\dots ,k_r\le n$ and
$k_1+\cdots +k_r<rn$. This implies that $\deg \ell_1 \cdots \ell_m \ne \deg y_n^N$ so $y_n^N=0$, a contradiction. Thus $y_n$ has infinite height and hence $N\beta_n$ is a root of $\NA (V)$.
\end{proof}

\begin{lem} \label{le:um^2}
Let $m\in \N _0$ with $y_m\ne 0$.
Then $y_m^2=0$ if and only if $p_m=-1$ and $y_{m+1}=0$.
\end{lem}

\begin{proof}
If $y_m^2=0$, then Lemma \ref{le:yy2} says that $y_{m+1}=0$ and $p_m=-1$. 

Conversely, assume that $y_{m+1}=0$, $y_m\ne 0$, and $p_m=-1$. Then
\begin{align*}
\partial_2(y_m^2)&=\mu_m
(y_mx_1^m+x_1^mq_{21}^mq_{22}y_m)\\
&=\mu_m(1+(q_{11}^mq_{12})^mq_{21}^mq_{22})y_mx_1^m=0.
\end{align*}
Since $\partial_1(y_m)=0$, we conclude that $y_m^2=0$.
\end{proof}

\begin{lem} \label{le:doubleroot}
Let $m\in \N _0$.
Assume that $p_m=-1$ and $y_{m+1}\ne
0$. Then $2\beta_m$ is a root of $V$ and
$q_{2\beta_{m}\, 2\beta_m}=1$.
\end{lem}

\begin{proof}
As $y_{m+1}\ne0$, Lemma \ref{le:um^2} says that $y_m^2\ne 0$. Now Lemma \ref{le:yNneq0isroot} applies and $2\beta_m$ is a root. Finally, 
$q_{2\beta_{m}\, 2\beta_m}=p_m^4=1$.
\end{proof}

\subsubsection{On the $w_m$'s}\label{subsubsec:wm}
We consider the following elements of $\NA (V)$:
\begin{align} 
\label{eq:wm}
w_m&=y_{m+2}y_m-q_{\beta_{m+2},\beta_m} y_my_{m+2}, &  m &\in \N _0.
\end{align}
Notice that $w_m$ is $\N _0^2$-homogeneous of degree $2\beta_{m+1}$.

\medbreak
Let $m\in \N _0$. Assume that $y_{m+2}\ne 0$ and $q_{\beta_{m+1} \, \beta_{m+1}}\ne -1$.
Let
\begin{align}\label{eq:def-w-tilde}
\wt_m &=
w_m-\frac {q_{\beta_{m+1} \, \beta_m}
(m+2)_{q_{11}}(1-q_{11}^{m+1}\qt_{12})}{1+
q_{\beta_{m+1} \, \beta_{m+1}}}y_{m+1}^2.
\end{align}

Our next goal is to determine when $\wt_m=0$.

\begin{lem} \label{le:wm}
Let $m\in \N _0$. Assume that 
\begin{align*}
y_{m+2}&\ne 0,&  p_{m+1}&\ne -1, &
w_k &\in \Bbbk y_{k+1}^2, \quad 0 \le k<m.
\end{align*}
Then the following are equivalent:
\begin{enumerate}
\item\label{item:wm-1} $w_m\in \Bbbk y_{m+1}^2$.
\item\label{item:wm-2} $\wt_m=0$.
\item\label{item:wm-3} $\partial_1^m\partial_2 (\wt_m)=0$.
\item\label{item:wm-4} The following equation holds: 
\begin{align*}
0= \Big(1- \frac{p_{m+1}}{q_{11}} \Big) \Big( 1+ \frac{p_{m+1}}{q_{11}} + \frac{p_m
	(m+2)_{q_{11}}^!(1-q_{11}^m\qt_{12})(1-q_{11}^{m+1}\qt_{12})}{(m)_{q_{11}}^!(1+q_{11})(1+p_{m+1})}\Big).
\end{align*}
\end{enumerate}
\end{lem}

\begin{proof}
First we compute
\begin{align*} 
&\partial_1^m\partial_2(w_m)=
\partial_1^m\Big( \mu_{m+2}x_1^{m+2}q_{21}^mq_{22}y_m
+y_{m+2}\mu_{m}x_1^m\Big)\\
&\quad-q_{\beta_{m+2} \, \beta_{m}}\Big(
\mu_{m}x_1^mq_{21}^{m+2}q_{22}y_{m+2}
+y_m\mu_{m+2}x_1^{m+2} \Big)\\
&=\mu_{m+2}q_{\beta_{m} \, \beta_{m}}\frac{(m+2)^!_{q_{11}}}{1+q_{11}} (x_1^2y_m- q_{11}^{2m}q_{12}^2y_mx_1^2)
\\
&\quad +\mu_{m} (m)^!_{q_{11}}(1-q_{\beta_{m+2} \, \beta_{m}}
q_{\beta_{m} \, \beta_{m+2}})y_{m+2}.
\end{align*}
Similarly, we compute
\begin{align*}
\partial_1^m\partial_2(y_{m+1}^2) &= \mu_{m+1}(m+1)^!_{q_{11}}
\big(q_{\beta_{m} \, \beta_{m+1}}y_{m+2}
+(1+p_{m+1})y_{m+1}x_1\big),
\end{align*}
see also the calculation in the proof of Lemma~\ref{le:um^2}.

\eqref{item:wm-1}$\implies $\eqref{item:wm-2}. Using the previous formulas one obtains quickly that
\begin{align*}
\partial_1^{m+1}\partial_2(w_m)&=\mu_{m+2} q_{\beta_{m+1} \, \beta_m}(m+2)^!_{q_{11}}y_{m+1},
\\
\partial_1^{m+1}\partial_2(y_{m+1}^2) &= \mu_{m+1}(m+1)^!_{q_{11}}
(1+p_{m+1})y_{m+1}.
\end{align*}
Since $y_{m+1}\ne 0$, this implies the claim.

\eqref{item:wm-2}$\implies $\eqref{item:wm-3}. Trivial.

\eqref{item:wm-3}$\iff $\eqref{item:wm-4}.
We notice that
$$x_1^2y_m-q_{11}^{2m}q_{12}^2y_mx_1^2
=y_{m+2}+(1+q_{11})q_{11}^mq_{12}y_{m+1}x_1.$$
Since $y_{m+2}$ and $y_{m+1}x_1$ are linearly independent in $\NA (V)$,
the formulas at the beginning of the proof imply that \eqref{item:wm-3} is equivalent to
\begin{align*}
&\frac{p_m(m+1)_{q_{11}}(m+2)_{q_{11}}(1-q_{11}^m\qt_{12})(1-q_{11}^{m+1}\qt_{12})}{1+q_{11}}
+1-q_{11}^{-2}p_{m+1}^2\\
&-\frac{q_{11}^{-1}p_mp_{m+1}(m+1)_{q_{11}}(m+2)_{q_{11}}(1-q_{11}^m\qt_{12})(1-q_{11}^{m+1}\qt_{12})}{1+p_{m+1}}
=0.
\end{align*}
This gives the equivalence between \eqref{item:wm-3} and \eqref{item:wm-4}.

\eqref{item:wm-3}$\implies $\eqref{item:wm-1}. 
We prove by induction on $k$ that
\begin{align} \label{eq:delkwm}
\partial_1^{m-k}\partial_2(\wt_m)&=0 & \text{for all }& 0\le k\le m.
\end{align}
If so, then $\partial_i(\wt_m)=0$ for $i = 1,2$
and then $\wt_m=0$.

For $k=0$, \eqref{eq:delkwm} holds by assumption. Now let $k>0$ and assume the statement holds for $j<k$. Notice that
\begin{align*}
\ker \partial _1 & \cap \toba^{(m+k+2)\alpha_1+\alpha_2} =\Bbbk y_{m+k+2}
\\ &\quad \implies  \, \partial_1^{m-k}\partial_2(\wt_m) = b\, y_{m+k+2} 
\text{ for some }b\in\Bbbk.
\end{align*}
We may assume that $y_{m+k+2}\neq 0$, otherwise the induction step holds.

Let $\iota:V\to V^*$ be the linear isomorphism with $x_i\mapsto \partial_i$
for $i\in \I$. Then $\iota$ is an isomorphism of braided vector
spaces and hence induces an isomorphism between the Hopf algebras 
$\NA(V)$ and $\NA (V^*)$ in $\ydh$. Notice that
\begin{align}\label{eq:derivation-wm}
\iota(y_{m-j})(\wt_m)&=\partial_1^{m-j}\partial_2(\wt_m)=0 & 
\mbox{for all }&0\le j<k,
\end{align}
since $y_j=x_1^jx_2+$ terms ending in $x_1$ and 
$\partial_1(\wt_m)=0$. Hence
\begin{align*}
\iota(w_{m-k})(\wt_m) &= \iota(y_{m-k+2}y_{m-k}-q_{\beta_{m-k+2},\beta_{m-k}} y_{m-k}y_{m-k+2})(\wt_m)
\\
&= \iota(y_{m-k+2})\iota(y_{m-k})(\wt_m) 
= b \iota(y_{m-k+2})(y_{m+k+2})
\\
&= b \partial_1^{m-k+2}\partial_2 (y_{m+k+2}) 
= b\mu_{m+k+2}\frac{(m+k+2)_{q_{11}}^!}{(m-k+2)_{q_{11}}^!}.
\end{align*}
On the other hand, using that $w_{m-k}=a_{m-k}y_{m-k+1}^2$ for some $a_{m-k}\in\Bbbk$ and the inductive hypothesis,
\begin{align*}
\iota(w_{m-k})(\wt_m) &= a_{m-k}\iota(y_{m-k+1}^2)(\wt_m)
= a_{m-k}\iota(y_{m-k+1})\partial_1^{m-k+1}\partial_2(\wt_m)=0.
\end{align*}
Hence $b=0$, which completes the inductive step.
\end{proof}

\begin{remark}\label{rem:wm-bis}
In the following cases, the right hand side of the equation in Lemma~\ref{le:wm} \eqref{item:wm-4} is equal to the following:
\begin{enumerate}[leftmargin=*,label=\rm{(\roman*)}]
\item\label{item:wm-bis-1} $m=0$: $(1+q_{22})(1-\qt_{12}q_{22})(1+q_{11}\qt_{12}^{\, 2}q_{22})
(1+q_{11}\qt_{12}q_{22})^{-1}$.
\item\label{item:wm-bis-2} $m=1$, $q_{22}=-1$:
$$(1+q_{11}^3\qt^{\, 2}_{12})
(1-q_{11}^3\qt_{12})(3)_{-q_{11}\qt_{12}}(1+q_{11}^2\qt_{12})^{-1}.$$
\item\label{item:wm-bis-3} $m=1$, $\qt_{12}q_{22}=1$:
$$ (1-q_{11}^3\qt_{12})(4)_{q_{11}}(3)_{-q_{11}^2\qt_{12}}
(1+q_{11}^4\qt_{12})^{-1}.
$$
\item\label{item:wm-bis-4} $m=1$, $q_{11}\qt_{12}^{\, 2}q_{22}=-1$:
$$ (1+q_{11}^2)(1-\qt_{12}^{\, \, -1})(1-q_{11}^3\qt_{12})(1-q_{11})^{-1}.
$$
\item\label{item:wm-bis-5} $m=2$, $q_{22}=-1$:
\begin{align*}
& (1+q_{11}^8\qt_{12}^{\, 3})(1-q_{11}^4\qt_{12})
(3)_{q_{11}^3\qt_{12}}^{-1}\cdot\\
&\quad \big( q_{11}^{10}\qt_{12}^{\, 4}+(q_{11}^7+q_{11}^6)\qt_{12}^{\, 3}
-(3)_{q_{11}}q_{11}^4\qt_{12}^{\, 2}+(q_{11}^4+q_{11}^3)\qt_{12}+1\big).
\end{align*}
\item\label{item:wm-bis-6} $m=2$, $\qt_{12}q_{22}=1$, $q_{11}^2=-1$:
$1-\qt_{12}^4$.
\item\label{item:wm-bis-7} $m=2$, $\qt_{12}q_{22}=1$, $(3)_{-q_{11}^2\qt_{12}}=0$:
$$ (1+q_{11}^4\qt_{12})(1-q_{11}^4\qt_{12})(1-q_{11}^5\qt_{12})(5)_{q_{11}}(3)_{-q_{11}}
(1+q_{11}^9\qt_{12}^{\, 2})^{-1}.
$$
\end{enumerate}
\end{remark}

Recall the definition of $\wt_m$, $m\in\N_0$, given in \eqref{eq:def-w-tilde}.
We study in the next Lemmas when $\wt_m\neq 0$ for small values of $m$.

\begin{lem}\label{lem:conditions-w0=0}
Assume that $y_2\ne 0$, $p_1\neq -1$. Then $\wt_0=0$ if and only if
$$ (\qt_{12}q_{22}-1)(q_{22}+1)(q_{11}\qt_{12}^{\, 2}q_{22}+1)=0. $$
\end{lem}

\begin{proof}
The claim follows by Lemma \ref{le:wm} and Remark \ref{rem:wm-bis} \ref{item:wm-bis-1}.
\end{proof}

Next we give conditions on the matrix $\bq$ which are equivalent to the equation $\wt_1=0$.

\begin{lem}\label{lem:conditions-w1=0}
Assume that $y_3\ne 0$, $p_2\neq -1$.
\begin{enumerate}[leftmargin=*,label=\rm{(\alph*)}]
\item\label{item:conditions-w1=0-1} If $q_{22}=-1$, then $\wt_1=0$ if and only if
\begin{align*}
(1-q_{11}^3\qt_{12})(q_{11}^3\qt_{12}^{\, 2}+1)(3)_{-q_{11}\qt_{12}}=0.
\end{align*}
\item\label{item:conditions-w1=0-2} If $\qt_{12}q_{22}=1$, then $\wt_1=0$ if and only if
\begin{align*}
(1-q_{11}^3\qt_{12})(q_{11}^2+1)(3)_{-q_{11}^2\qt_{12}}=0.
\end{align*}
\item\label{item:conditions-w1=0-3} If $q_{11}\qt_{12}^{\, 2}q_{22}=-1$, then $\wt_1=0$ if and only if
$(1-q_{11}^3\qt_{12})(q_{11}^2+1)=0$.
\end{enumerate}
\end{lem}

\begin{proof}
As we assume $y_3\neq 0$, we have that $(3)_{q_{11}}^{!}\mu_3\ne 0$. That is,
\begin{align*}
q_{11}&\notin \G_2\cup\G_3, & q_{11}^k\qt_{12} &\neq 1, \quad k=0,1,2.
\end{align*}

Assume that $q_{22}=-1$, so we have that $p_1=-q_{11}\qt_{12}\neq -1$. Hence $\wt_0=0$ by Lemma \ref{lem:conditions-w0=0}. Now \ref{item:conditions-w1=0-1} follows by Lemma \ref{le:wm} and Remark \ref{rem:wm-bis} \ref{item:wm-bis-2}.

\smallbreak

To prove \ref{item:conditions-w1=0-2}, we assume that $\qt_{12}q_{22}=1$. Again we have that $p_1=q_{11}\neq -1$, so $\wt_0=0$ by Lemma \ref{lem:conditions-w0=0}. Now we can apply Lemma \ref{le:wm} and Remark \ref{rem:wm-bis} \ref{item:wm-bis-3} and the claim follows.

\smallbreak

Finally, if $q_{11}\qt_{12}^{\, 2}q_{22}=-1$, then $p_1=-\qt_{12}^{-1}\neq -1$, so $\wt_0=0$ by Lemma \ref{lem:conditions-w0=0}. Hence \ref{item:conditions-w1=0-3} follows by Lemma \ref{le:wm} and Remark \ref{rem:wm-bis} \ref{item:wm-bis-4}.
\end{proof}

\subsection{Proof of Theorem \ref{th:infGK}}\label{subsec:proof-theta-2}

First we extend Lemma \ref{le:dimindefCartan} to any braided vector space of diagonal type and dimension two.

\begin{pro} \label{pro:rootweight1}
Assume that $V$ is of dimension two and $\qt_{12}\neq 1$.
If there is a root $\gamma $ of $V$ such that $q_{\gamma \,\gamma}=1$,
then $\GK \NA(V)=\infty$.
\end{pro}

\begin{proof} 
If $V$ does not admit all reflections, then $\GK \toba(V)=\infty$ by Remark \ref{rem:all-reflections-gkd}. 

Now we assume that $V$ admits all reflections.
If $V$ is generic, then $V$ is of Cartan type. As $\gamma$ cannot be a real root since $q_{\gamma \, \gamma}=1$, $V$ is not of finite type. 
Then $\GK \toba(V)=\infty$ by Remark \ref{rem:nichols-diagonal-finite-gkd} \eqref{item:rosso-aa}. 

If $V$ is semigeneric, then $\GK \toba(V)=\infty$. Indeed if we suppose that $\GK \toba(V)<\infty$, then all roots $\beta$ satisfy $q_{\beta \, \beta} \neq 1$ since the root system is finite by Corollary \ref{coro:infGK-semigeneric}, and this gives a contradiction.

Finally, if $V$ is of torsion class,
then the set $\X$  is finite by Remark \ref{rem:nichols-diagonal-finite-gkd} \eqref{item:torsion-X-finite}. 
If the orbit of $\gamma$ is infinite, then $\GK \toba(V)=\infty$ since 
there are infinitely many roots $\delta$ of $V$ with $q_{\delta,\delta}=1$.
Now assume that $\gamma$ has finite orbit. 
Let $s_1,s_2$ be the simple reflections corresponding to $V$, cf. \eqref{eq:def-si}. Then there exists
$k>0$ such that $(s_1s_2)^k(\gamma)=\gamma$ and $(\cR^1\cR^2)^k(V)=V$. Thus $1$ is an eigenvalue of $(s_1s_2)^k \in \Aut(\Z^2)$. 
Since $\det(s_1s_2)=1$, the other eigenvalue is also $1$. Thus either $(s_1s_2)^k=\id$, or else $(s_1s_2)^k$ is a
shear mapping. The first case implies that the set of real roots is finite, hence the Weyl groupoid is finite by \cite{CH}, and consequently
all roots $\delta$ of $V$ are real \cite{CH}, a contradiction. 

In the second case there exist $c_i\in\Z$ such
that $(s_1s_2)^k(\alpha_i)=\alpha_i+c_i\gamma$ for $i= 1,2$. 
Hence, each $\beta_n:=(s_1s_2)^{nk}(\alpha_i)=\alpha_i+c_i n\gamma$ is a real root for $n\in\N$
and $\GK\toba(V)=\infty$ by Lemma \ref{lemma:rosso-lemma19-gral}.
\end{proof}

As a consequence of Proposition \ref{pro:rootweight1} we have:  
\begin{coro} \label{cor:crucialGK1}
Let $p\in \Bbbk ^\times $ such that $p^4\ne 1$. Assume that $q_{11}=p$ and
$\qt_{12}=q_{22}=p^4$. Then $\GK \NA(V)=\infty$.
\end{coro}

\begin{proof}
If $p$ is a root of $1$ and $p^4\ne 1$, then $V$ is of Cartan type. Let $A=(a_{ij})_{i,j\in\I_2}$ be the Cartan matrix of $V$. 
Recall that $m\alpha_1+\alpha_2$ is a root if and only if $0\le m\le -a_{12}$.
We study three cases according with the order of $p$:

\begin{itemize}[leftmargin=*]
\item If $p \in \G_{4N}'$, $N\ge 2$, then $-a_{12}=4N-4$. Thus $\gamma=\beta_{2N-2}$ is a root. As
\begin{align*}
q_{\gamma \, \gamma}&=q_{11}^{(2N-2)^2}\qt_{12}^{\, \, 2N-2}q_{22}
= p^{4N^2}=1,
\end{align*}
$\GK \NA(V)=\infty$ by Proposition \ref{pro:rootweight1}.

\item If $p \in \G_{4N+2}'$, $N\ge 1$, then $-a_{12}=4N-2$.
In particular, $\gamma=\beta_{2N-1}$ is a root.
Now $2\beta_{2N-1}$ is a root by Lemma~\ref{le:doubleroot}, since 
$y_{2N}\neq 0$ and
\begin{align*}
p_{2N-1}=q_{\gamma \, \gamma}&=q_{11}^{(2N-1)^2}\qt_{12}^{\, \, 2N-1}q_{22}
= p^{(2N+1)^2}=-1.
\end{align*}
Thus $\GK \NA(V)=\infty$ by Proposition \ref{pro:rootweight1}.

\item If $p \in \G_{2N+1}'$, $N\ge 1$, then $-a_{21}=2N$. 
Thus $\gamma=N\alpha_2+\alpha_1$ is a root, and
\begin{align*}
q_{\gamma \, \gamma}&=q_{22}^{N^2}\qt_{12}^{\, \, N}q_{11}
= p^{(2N+1)^2}=1,
\end{align*}
$\GK \NA(V)=\infty$, again by Proposition \ref{pro:rootweight1}.
\end{itemize}
\smallbreak

Finally, if $p$ is not a root of $1$, then $\NA (V)$ does not admit all reflections and hence $\GK \NA(V)=\infty$.
\end{proof}

\bigbreak We apply next Corollary \ref{cor:crucialGK1} in the braided Hopf algebra $K_{\ge m+1}/K_{>m+1}$.
Since $x_1^k\in K_{>m}$ for all $k,m\in \N $,
\eqref{eq:copr_um} implies that $y_m\in K_{\ge m}$ is a primitive element of $\N _0^2$-degree $\beta_m$ in $K_{\ge m}/K_{>m}$ for all $m\in \N $.
Then \eqref{eq:copr_um} and \eqref{eq:copr_um} leads to the following shape
of the coproduct of $w_m$ and $y_m^2$:
\begin{align} \label{eq:copr_wm_simple}
&
\begin{aligned}
\Delta (w_m)&=w_m\ot 1+1\ot w_m
\\
&+(m+2)_{q_{11}}(1-q_{11}^{m+1}\qt_{12})q_{\beta_{m+1}\,\beta_m} y_{m+1}\ot y_{m+1}
\\
&
+\text{ terms }x\otimes y, \, \deg x=k\alpha_1+l\alpha _2, \, k\ge l(m+1)+1;
\end{aligned}
\\
\label{eq:copr_um2}
&
\begin{aligned}
\Delta (y_m^2)&\in y_m^2\ot 1+(1+p_m)y_m\ot y_m
+1\ot y_m^2+B_{>m}\otimes \NA (V).
\end{aligned}
\end{align}
Assume that $p_{m+1}\neq -1$. 
By \eqref{eq:copr_wm_simple} and \eqref{eq:copr_um2}, $\wt_m$ is a primitive element of $\N _0^2$-degree $2\beta_{m+1}$
in $K_{\ge m+1}/K_{>m+1}$.

\begin{lem}\label{le:technical-0}
Let $m\in \N_0$ be such that $p_{m+1}\neq -1$. If $\wt_m\neq 0$, then $\GK \NA(V)=\infty$.
\end{lem}

\pf The subalgebra of $K_{\ge m+1}/K_{>m+1}$ generated by $y_{m+1}$ and $\wt_m$ is
a pre-Nichols algebra of diagonal type.
Let $W$ be the $\Bbbk $-span of $y_{m+1}$ and $\wt_m$, $p=q_{\beta_{m+1} \, \beta_{m+1}}$.
The braiding matrix of $W$ is
$$ \begin{pmatrix} p & p^2 \\ p^2 & p^4 \end{pmatrix} $$

If $p^4=1$ then $2\beta_{m+1}$ is a root of $V$ of infinite height. 
Thus $\GK \NA(V)=\infty$ by Proposition~\ref{pro:rootweight1}.

If $p^4\neq 1$, then $W$ satisfies the assumptions of 
Corollary~\ref{cor:crucialGK1} and hence $\GK\NA (W)=\infty$.
Since $\NA (W)$ is a subquotient of $K_{\ge m+1}/K_{>m+1}$, 
and this is a subquotient of $\NA(V)$, we have that $\GK \NA(V)=\infty$.
\epf

We now apply Lemma \ref{le:technical-0} combined with the Lemmas in \S \ref{subsection:rank-2}.

\begin{lem}\label{le:technical-1}
Assume that 
\begin{align*}
\qt_{12} &\ne 1, & (2)_{q_{11}}(1-q_{11}\qt_{12}) & \ne 0, & 
(2)_{q_{22}}(1-\qt_{12}q_{22}) & \ne 0, &
q_{11}\qt_{12}^{\, 2} q_{22} & \ne -1.
\end{align*}
Then $\GK \NA(V)=\infty$.
\end{lem}

\begin{proof}
By Lemma \ref{lemmata:rosso}, $y_2\ne 0$. If $p_1=-1$,
then $2\beta_1$ is a root of $V$ by Lemma~\ref{le:doubleroot}, so
$\GK \NA(V)=\infty$ by Proposition~\ref{pro:rootweight1}.

Assume now that $p_1\ne -1$. Then
$$ \Delta (y_1^2)=y_1^2\ot 1+1\ot y_1^2+(1+p_1)y_1\ot y_1
$$
in $K_{\ge 1}/K_{>1}$ and hence $\wt_0$ is primitive in $K_{\ge 1}/K_{>1}$. By Lemma \ref{lem:conditions-w0=0} $\wt_0$ is
non-zero. Now we apply Lemma \ref{le:technical-0}.
\end{proof}

Without loss of generality, we  assume that $|a^V_{12}|\ge |a^V_{21}|>0$. We find a bound for $a^V_{12}$, $a^V_{21}$ to reduce the possibilities.

\begin{lem}\label{le:aij-ge-3}
Let $V$ be a braided vector space of diagonal type and dimension 2 such that $a_{12},a_{21}\le -3$. Then $\GK \toba(V)=\infty$.
\end{lem}
\pf
Suppose that $\GK \toba(V)<\infty$. Then $\wt_0=0$, so by Lemma \ref{lem:conditions-w0=0} we have that $q_{11}\qt_{12}^{\, 2}q_{22}=-1$. As $y_3\neq 0$, Lemma \ref{lem:conditions-w1=0}\ref{item:conditions-w1=0-3} implies that $(1-q_{11}^3\qt_{12})(1+q_{11}^2)=0$, so $a_{12}=-3$ and analogously $a_{21}=-3$. We exclude the case $q_{11}^2=q_{22}^2=-1$: if this happens, then $1= q_{11}^2\qt_{12}^{\, 4}q_{22}^2= \qt_{12}^{\, 4}$, so $\qt_{12}\in\G_4$, but this gives a contradiction since 
$a_{12}^V=a_{21}^V=-3$.
Hence we may assume $q_{22}^3\qt_{12}=1$. 
Let $r=q_{22}$, so $\qt_{12}=r^{-3}$. As $-1=q_{11}\qt_{12}^{\, 2}q_{22}=r^{-5}q_{11}$, we have that $q_{11}=-r^5$.

First we assume that $q_{11}^3\qt_{12}=1$. Hence $q_{11}^3=r^3$ so either $q_{11}=r\in\G_8'$ or $q_{11}=r^{17}\in\G_{24}'$.
We compute $d_t$ as in \eqref{eq:dt} when $n=1$, for each case:
\begin{itemize}[leftmargin=*]
	\item If $q_{11}=r\in\G_8'$, then $p_1=r^7$, so $N=8$ and
	\begin{align*}
	d_t &= 1-r^{7t+6} +\frac{r^{7t}(1-r^{-2})(1+r)}{(t)_{r^7}}
	\\
	&= 1-r^{7t+6} +\frac{r^{7t-3}(r^2-1)^2}{1-r^{7t}}
	= 1-r^{7t+6} +\frac{2r^{7t+3}}{1-r^{7t}}.
	\end{align*}
	Hence $d_t=0 \iff (1-r^{7t+6}) (1-r^{7t})= 2r^{7t+7}$.
	Now we check the validity of this equation for $1\le t\le N-2=6$:
	\begin{align*}
	t&=1: & (1-r^{5}) (1-r^{7}) &\neq 2r^{6};
	\\
	t&=2: & (1-r^{4}) (1-r^{6})=2(1-r^6) &\neq 2r^{5};
	\\
	t&=3: & (1-r^{3}) (1-r^{5}) &\neq  2r^{4};
	\\
	t&=4: & (1-r^{2}) (1-r^{4}) =2(1-r^2) &\neq 2r^{3};
	\\
	t&=5: & (1-r) (1-r^{3}) & \neq  2r^{2};
	\\
	t&=6: & (1-1) (1-r^{2})=0 &\neq 2r.
	\end{align*}
	Thus $d_t\neq 0$ for all $1\le t\le 6$ so $y_1^8\neq 0$ by Proposition \ref{prop:dt-yn}, and hence $8\beta_1$ is a root by Lemma \ref{le:yNneq0isroot}. This implies that $\GK \toba(V)=\infty$ by Proposition \ref{pro:rootweight1}.
	\medspace
	
	\item If $q_{11}=r^{17}\in\G_{24}'$, then $p_1=r^{15}$, so $N=8$ and
	\begin{align*}
	d_t=1-r^{15t-2}+\frac{r^{15t}(1+r^2)(1+r^{17})(1+r^3)}{1-r^{15t}}.
	\end{align*}
	Hence $d_t=0 \iff (1-r^{15t-2}) (1-r^{15t})= -r^{15t}(1+r^2)(1+r^{17})(1+r^3)$.
	Now we check the validity of this equation for $1\le t\le N-2=6$:
	\begin{align*}
	t&=1: & (1-r^{13}) (1-r^{15}) & \neq r^{3}(1+r^2)(1+r^{17})(1+r^3) ;
	\\
	t&=2: & (1-r^{4}) (1-r^{6}) & \neq r^{18}(1+r^2)(1+r^{17})(1+r^3);
	\\
	t&=3: & (1-r^{19}) (1-r^{21}) & \neq r^{9}(1+r^2)(1+r^{17})(1+r^3) ;
	\\
	t&=4: & 2(1-r^{10}) & \neq (1+r^2)(1+r^{17})(1+r^3) ;
	\\
	t&=5: & (1-r) (1-r^{3}) & \neq r^{15}(1+r^2)(1+r^{17})(1+r^3) ;
	\\
	t&=6: & (1-r^{16}) (1-r^{18}) & \neq r^{6}(1+r^2)(1+r^{17})(1+r^3) .
	\end{align*}
	Thus $d_t\neq 0$ for all $1\le t\le 6$ so $y_1^8\neq 0$ by Proposition \ref{prop:dt-yn}, and hence $8\beta_1$ is a root by Lemma \ref{le:yNneq0isroot}. This implies that $\GK \toba(V)=\infty$ by Proposition \ref{pro:rootweight1}.
\end{itemize}

The last case is $q_{11}\in \G_4'$. Thus $r^5=-q_{11}$. As $a_{12}^V=-3$ we have  that $r\in\G_{20}'$.
Again we compute $d_t$ for $n=1$. Here,
$p_1=r^{13}$, so $N=20$ and
\begin{align*}
d_t & = 1-r^{13t} +\frac{r^{13t}(1+r^{2})(1+r^{15})(1+r^{3})}{1-r^{13t}}.
\end{align*}
Hence $d_t=0 \iff (1-r^{13t})^2 = r^{13t+5}(1+r^{2})(1+r^{5})(1+r^{3})$.
Now we check the validity of this equation for $1\le t\le N-2=18$:
\begin{align*}
t&=1: & (1+r^{3})^2 &\neq r^{18}(1+r^{2})(1+r^{5})(1+r^{3}) ;
\\
t&=2: & (1-r^{6})^2 &\neq r^{11}(1+r^{2})(1+r^{5})(1+r^{3});
\\
t&=3: & (1+r^{9})^2 &\neq r^{4}(1+r^{2})(1+r^{5})(1+r^{3}) ;
\\
t&=4: & (1+r^{2})^2 &\neq r^{17}(1+r^{2})(1+r^{5})(1+r^{3}) ;
\\
t&=5: & 2r^{15} &\neq r^{10}(1+r^{2})(1+r^{5})(1+r^{3});
\\
t&=6: & (1+r^{8})^2 &\neq r^{3}(1+r^{2})(1+r^{5})(1+r^{3});
\\
t&=7: & (1+r)^2 &\neq r^{16}(1+r^{2})(1+r^{5})(1+r^{3}) ;
\\
t&=8: & (1-r^{4})^2 &\neq r^{9}(1+r^{2})(1+r^{5})(1+r^{3}) ;
\\
t&=9: & (1+r^{7})^2 &\neq r^{2}(1+r^{2})(1+r^{5})(1+r^{3}) ;
\\
t&=10: & 4 &\neq r^{15}(1+r^{2})(1+r^{5})(1+r^{3}) ;
\\
t&=11: & (1-r^{3})^2 &\neq r^{8}(1+r^{2})(1+r^{5})(1+r^{3}) ;
\\
t&=12: & (1+r^{6})^2 &\neq r(1+r^{2})(1+r^{5})(1+r^{3}) ;
\\
t&=13: & (1-r^{9})^2 &\neq r^{14}(1+r^{2})(1+r^{5})(1+r^{3}) ;
\\
t&=14: & (1-r^{2})^2 &\neq r^{7}(1+r^{2})(1+r^{5})(1+r^{3}) ;
\\
t&=15: & 2r^{5} &\neq (1+r^{2})(1+r^{5})(1+r^{3}) ;
\\
t&=16: & (1-r^{8})^2 &\neq r^{13}(1+r^{2})(1+r^{5})(1+r^{3}) ;
\\
t&=17: & (1-r)^2 &\neq r^{6}(1+r^{2})(1+r^{5})(1+r^{3}) ;
\\
t&=18: & (1+r^{4})^2 &\neq r^{19}(1+r^{2})(1+r^{5})(1+r^{3}).
\end{align*}
Thus $d_t\neq 0$ for all $1\le t\le 18$ so $y_1^{20} \neq 0$ by Proposition \ref{prop:dt-yn}, and hence $20\beta_1$ is a root by Lemma \ref{le:yNneq0isroot}. Again, $\GK \toba(V)=\infty$ by Proposition \ref{pro:rootweight1}.
\epf

We finally assume that $\NA (V)$ has finite GK-dimension.
By Remark \ref{rem:all-reflections-gkd}, $V$ admits all reflections. 
We consider all possible cases with $a_{21}\in\{-1,-2\}$, not covered by previous arguments,
and conclude that the root system is finite--i.~e. the Dynkin diagram appears in \cite[Table 1]{H-classif}. 

\subsubsection{$a^V_{12}=a^V_{21}=-1$}\label{subsubsec:a12=a21=-1}
We have that 
\begin{align*}
(q_{11}\qt_{12}-1)(2)_{q_{11}}&=0, &
(q_{22}\qt_{12}-1)(2)_{q_{22}}&=0.
\end{align*}
The four possible diagrams appear in \cite[Table 1, Rows 1 \& 2]{H-classif}.

\subsubsection{$a^V_{12}=-2$, $a^V_{21}=-1$}\label{subsubsec:a12=-2-a21=-1}
We have that 
\begin{align*}
(q_{11}^2\qt_{12}-1)(3)_{q_{11}}&=0, &
(q_{22}\qt_{12}-1)(2)_{q_{22}}&=0.
\end{align*}
If $q_{22}\qt_{12}=1$, then we have \cite[Table 1, Rows 3 \& 5]{H-classif}. 

Now we assume that $q_{22}=-1$. If $q_{11}^2\qt_{12}=1$, then we get \cite[Table 1, Row 4]{H-classif}. Let $q_{11}\in\G_3'$. For simplicity we set $q=q_{11}$, $r=\qt_{12}$. Let $(t_{ij})_{i,j\in\I}$ be the braiding matrix of $\cR^2(V)$: its Dynkin diagram is
$\xymatrix{\overset{-qr} {\circ} \ar@{-}[r]^{r^{-1}} & \overset{-1}{\circ}}$.
We study the possible values of $a:=a_{12}^{\cR^2(V)}$. Notice that $a\leq -2$.

\begin{itemize} [leftmargin=*]\renewcommand{\labelitemi}{$\circ$}
\item $a=-2$. If $1=t_{11}^2\widetilde{t}_{12}=q^2r$, then $V$ appears in \cite[Table 1, Row 4]{H-classif}. Otherwise $1=t_{11}^3=-r^3$, so  $r=-q^{\pm1}$. As $1\neq t_{11}=-qr$, we have that $r=-q$ and $V$ appears in \cite[Table 1, Row 6]{H-classif}.
\smallbreak

\item $a=-3$. Either $1=t_{11}^3\widetilde{t}_{12}=-r^2$, in which case $r\in\G_4'$ and $V$ appears in \cite[Table 1, Row 8]{H-classif}, or else 
$1=t_{11}^4=qr^4$, in which case $r\in\G_{12}'$ and $V$ appears in \cite[Table 1, Row 7]{H-classif}.
\smallbreak

\item $a\le -4$. Notice that $\wt_0=0$ by Lemma \ref{lem:conditions-w0=0}, and $\wt_1=0$ by Lemma \ref{lem:conditions-w1=0} \ref{item:conditions-w1=0-1} since $-t_{11}\widetilde{t}_{12}=q\in\G_3'$. Hence we may apply Lemma \ref{le:wm}: as $\wt_2=0$, the scalar in Remark \ref{rem:wm-bis} \ref{item:wm-bis-5} is zero. That is,
\begin{align*}
0 & = (1+q^2r^5) (1-qr^3)(1+q^2r^2)(1+q^2r^4).
\end{align*}
If $q^2r^5=-1$, then $-r\in\G_{15}'$ and $V$ belongs to \cite[Table 1, Row 15]{H-classif}. 
If $qr^3=1$, then $r\in\G_{9}'$ and $V$ is in \cite[Table 1, Row 9]{H-classif}.
If $q^2r^2=-1$, then $r\in \G'_{12}$, $q=-r^2$ and $V$ belongs to 
\cite[Table 1, Row 7]{H-classif}. Otherwise $q^2r^4=-1$, in which case $r\in \G'_{24}$ with $q=-r^4$, and $V$ is in \cite[Table 1, Row 12]{H-classif}.
\end{itemize}

\subsubsection{$a^V_{12}=-3$, $a^V_{21}=-1$ with $q_{22}=-1$}\label{subsubsec:a12=-3-q22=-1}

First we assume $q_{11}^3\qt_{12}=1$. Set $q=q_{11}$, so $\qt_{12}=q^{-3}$.
Let $(t_{ij})_{i,j\in\I}$ be the braiding matrix of $\cR^2(V)$: its Dynkin diagram is
$\xymatrix{\overset{-q^{-2}} {\circ} \ar@{-}[r]^{q^3} & \overset{-1}{\circ}}$.

\begin{itemize} [leftmargin=*]\renewcommand{\labelitemi}{$\circ$}
\item If $a_{12}^{\cR^2(V)}=-2$, $\cR^2(V)$ appears in \S \ref{subsubsec:a12=-2-a21=-1}, so $\cR^2(V)$ has finite root system and then $V$ too.

\item If $a_{12}^{\cR^2(V)}=-3$, then then either $t_{11}\in\G_4'$, in which case $q\in\G_8'$ and $V$ belongs to \cite[Table 1, Row 11]{H-classif}; or else $t_{11}^3\widetilde{t}_{12}=1$, in which case $q\in\G_6'$ and $V$ is of Cartan type $G_2$ \cite[Table 1, Row 10]{H-classif}.

\item If $a_{12}^{\cR^2(V)}\le-4$, then Remark \ref{rem:wm-bis} \ref{item:wm-bis-5} says that $(1+q^7)(1-q^5)(5)_{-q^{-2}}=0$.
If $q^7=-1$, then $V$ belongs to \cite[Table 1, Row 16]{H-classif}.
If $q^5=1$, then $V$ belongs to \cite[Table 1, Row 13]{H-classif}.
If $-q^2\in\G_5'$, then $V$ belongs to \cite[Table 1, Row 14]{H-classif}.
\end{itemize}

Finally, if $q_{11}^3\qt_{12}\neq 1$, then $q_{11}=\eta\in\G_4'$. Set $q=\qt_{12}$. By Lemma \ref{lem:conditions-w1=0}\ref{item:conditions-w1=0-1},

\begin{itemize} [leftmargin=*]\renewcommand{\labelitemi}{$\circ$}
\item either $\eta^3q^2=-1$, so $q^2=-\eta$ and $V$ belongs to \cite[Table 1, Row 11]{H-classif};

\item or else $-\eta q\in\G_3'$, in which case $a_{12}^{\cR^2(V)}=-2$, $a_{21}^{\cR^2(V)}=-1$ since the diagram of $\cR^2(V)$ is $\xymatrix{\overset{-\eta q} {\circ} \ar@{-}[r]^{q^{-1}} & \overset{-1}{\circ}}$. Hence $\cR^2(V)$ appears in \S \ref{subsubsec:a12=-2-a21=-1} and thus $\cR^2(V)$ has finite root system.
\end{itemize}

\subsubsection{$a^V_{12}=-3$, $a^V_{21}=-1$ with $q_{22}\qt_{12}=1$}\label{subsubsec:a12=-3-q22neq-1}

If $q_{11}^3\qt_{12}=1$, then $V$ is of Cartan type $G_2$ and the root system is finite \cite[Table 1, Row 10]{H-classif}. Otherwise, $q_{11}=\eta\in\G_4'$. For simplicity we call $q=q_{22}$ so $\qt_{12}=q^{-1}$.
Let $(t_{ij})_{i,j\in\I}$ be the braiding matrix of $\cR^1(V)$: its Dynkin diagram is
$\xymatrix{\overset{\eta} {\circ} \ar@{-}[r]^{-q} & \overset{q^{-2}\eta}{\circ}}$. By Lemma \ref{le:aij-ge-3}, $a_{21}^{\cR^1(V)}\ge -2$.

\begin{itemize} [leftmargin=*]\renewcommand{\labelitemi}{$\circ$}
\item If $a_{21}^{\cR^1(V)}=-1$, then either $t_{22}=-1$, in which case $q^2=-\eta$ and $V$ belongs to \cite[Table 1, Row 11]{H-classif}; or else $t_{22}\widetilde{t}_{12}=1$, in which case $q=-\eta$ and $V$ is of Cartan type $G_2$.

\item If $a_{21}^{\cR^1(V)}=-2$, then either $1=t_{22}^2\widetilde{t}_{12}$, in which case $q^3=1$,
or else $1=t_{22}^3$, which implies $q^6=-\eta$. For the first case, we compute $p_1=q^{-1}$,
which has order $N=3$, and by \eqref{eq:dt}, $d_1=-\eta q\neq 0$. Hence $y_1^3\neq 0$ by
Proposition~\ref{prop:dt-yn}. By Lemma \ref{le:yNneq0isroot}, $3\beta_1$ is a root of $\cR^1(V)$,
thus $\GK \toba(\cR^1(V))=\infty$ by Proposition \ref{pro:rootweight1}. For the second case,
$q\in\G_{24}'$ and $V$ belongs to \cite[Table 1, Row 12]{H-classif}.
\end{itemize}

\subsubsection{$a^V_{12}\le -4$, $a^V_{21}=-1$ with $q_{22}=-1$}\label{subsubsec:a12le-4-q22=-1}
For simplicity we set $q=q_{11}$, $r=\qt_{12}$. Let $(t_{ij})_{i,j\in\I}$ be the braiding matrix of $\cR^2(V)$: its Dynkin diagram is
$\xymatrix{\overset{-qr} {\circ} \ar@{-}[r]^{r^{-1}} & \overset{-1}{\circ}}$.
As $q^3r\neq 1$, Lemma \ref{lem:conditions-w1=0} \ref{item:conditions-w1=0-1} says that either $q^3r^2=-1$ or else $-qr\in\G_3'$.
\begin{itemize} [leftmargin=*]\renewcommand{\labelitemi}{$\circ$}
\item If $q^3r^2=-1$, then $t_{11}^3 \widetilde{t}_{12}=-q^3r^2=1$ so $a_{12}^{\cR^2(V)}= -3$. Hence $\cR^2(V)$ has a finite root system by \S \ref{subsubsec:a12=-3-q22=-1}, and $V$ too.

\item If $-qr\in\G_3'$, then $a_{12}^{\cR^2(V)} = -2$. Hence $\cR^2(V)$ has a finite root system by \S \ref{subsubsec:a12=-2-a21=-1}, and $V$ too.
\end{itemize}

\subsubsection{$a^V_{12}\le -4$, $a^V_{21}=-1$ with $q_{22}\qt_{12}=1$}\label{subsubsec:a12le-4-q22neq-1}
For simplicity set $q=q_{11}$, $r=q_{22}$, so $\qt_{12}=r^{-1}$. By Lemma \ref{lem:conditions-w1=0} \ref{item:conditions-w1=0-2}, $t:=-q^2r^{-1}\in\G_3'$ since $a_{12}^V\le -4$. As $y_4\ne0$ we may apply Lemma \ref{le:wm} and Remark \ref{rem:wm-bis} \ref{item:wm-bis-7}:
\begin{align*}
0=(1-q^4r^{-1})(1+q^4r^{-1})(1-q^5r^{-1})(3)_{-q}(5)_q.
\end{align*}
\begin{itemize} [leftmargin=*]\renewcommand{\labelitemi}{$\circ$}
\item If $q^4r^{-1}=1$, then $V$ is of affine Cartan type and $\GK \toba(V)=\infty$ by Proposition \ref{prop:dimaffineCartan}, a contradiction.

\item If $q^4r^{-1}=-1$, then $q^2 =(-q^4r^{-1})(-q^2r^{-1})^{-1}=t^2$, so $q=\pm t$, $r=-t$. If $q=t$, then $a_{12}^V\ge -2$; otherwise $q=-t$ and $q_{11}\qt_{12}=qr^{-1}=1$ so $a_{12}^V= -1$. In any case we obtain a contradiction with $a_{12}^V\le -4$.

\item If $q^5r^{-1}=1$, then $q^3=-t^{-1}\in\G_6'$. Hence $-q\in\G_9'$ and $p_3=q^9r^{-3}r=q^{-1}\in\G_{18}'$. Now we compute $d_l$ as in \eqref{eq:dt} for $n=3$:
\begin{align*}
d_l &= 1-q^{-l-1}q^6r^{-1}+\frac{q^{-l}(1-q^3r^{-1})(4)_{q}}{(l)_{q^{-1}}}
= 1-q^{-l}+\frac{q^{-l}(1-q^{-2})(4)_{q}}{(l)_{q^{-1}}}
\end{align*}
Thus $d_l=0$ if and only if $q^{l+3}(1-q^{-l})^{2}=(q^{2}-1)(q^{4}-1)$, but this equality does not hold for $1\le l\le 16$. Hence $18\beta_3$ is a root of $V$ by Proposition \ref{prop:dt-yn} and Lemma \ref{le:yNneq0isroot}, so $\GK \toba(V)=\infty$ by Proposition \ref{pro:rootweight1}.

\item If $-q\in\G_3'$, then either $t=-q$ or else $t=-q^{-1}$.
Both are not possible since $a^V_{12} \le -4$.

\item If $q\in\G_5'$, then $r=-t^{-1}q^2\in\G_{30}'$. This is the case in \cite[Table 1, Row 15]{H-classif}.
\end{itemize}

\subsubsection{$a^V_{12}=a^V_{21}=-2$}\label{subsubsec:a12=a21=-2}
We have that $(q_{11}\qt_{12}-1)(2)_{q_{11}}(q_{22}\qt_{12}-1)(2)_{q_{22}} \neq 0$,
\begin{align*}
(q_{11}^2\qt_{12}-1)(3)_{q_{11}}&=0, &
(q_{22}^2\qt_{12}-1)(3)_{q_{22}}&=0.
\end{align*}
If $q_{11}^2\qt_{12}=1=q_{22}^2\qt_{12}$, then $V$ is of affine Cartan type, a contradiction with Proposition \ref{prop:dimaffineCartan}. Hence we may assume that $q_{11}\in\G_3'$. By Lemma \ref{lem:conditions-w0=0}, $q_{11}\qt_{12}^{\, 2}q_{22}=-1$.
Let $(t_{ij})_{i,j\in\I}$ be the braiding matrix of $\cR^1(V)$:
\begin{align*}
t_{22}=q_{22}\qt_{12}^{\, 2} q_{11}^4 = q_{22}\qt_{12}^{\, 2} q_{11} =-1,
\end{align*}
so $\cR^1(V)$ appears in \S \ref{subsubsec:a12=-2-a21=-1} since $\GK \toba(\cR^1(V))<\infty$. Thus $\cR^1(V)$ has a finite root system, and $V$ too.

\subsubsection{$a^V_{12}\le-3$, $a^V_{21}=-2$}\label{subsubsec:a12le-3-a21=-2}

As $\wt_0=0$, $q_{11}\qt_{12}^{\, 2}q_{22}=-1$ by Lemma \ref{lem:conditions-w0=0}. As $\wt_1=0$, either $q_{11}^3\qt_{12}=1$ or else $q_{11}^2=-1$ by Lemma \ref{lem:conditions-w1=0} \ref{item:conditions-w1=0-3}. Hence $a_{12}=-3$. We analyze the possible 4 cases.

If $q_{11}^3\qt_{12}=1=q_{22}^2\qt_{12}$, then $V$ is of indefinite Cartan type and
\begin{align*}
-q_{11}^2q_{22}=(q_{11}\qt_{12}^{\, 2}q_{22})q_{11}^2q_{22} = q_{11}^3\qt_{12}q_{22}^2\qt_{12}=1,
\end{align*}
thus $q_{22}=-q_{11}^{-2}$, and $1=q_{22}^{2}\qt_{12}=q_{11}^{-7}$, so $q_{11}\in\G_7'$. Let $r:=q_{22}\in\G_{14}'$, so $q_{11}=r^{10}$, $\qt_{12}=r^{12}$. We compute $d_t$ as in \eqref{eq:dt} for $n=1$. Here, $p_1=r^{9}$, so $N=14$ and
\begin{align*}
d_t=1-r^{9t+13}
+\frac{r^{9t}(1+r)(1-r^3)(1+r^2)}{1-r^{9t}}.
\end{align*}
Hence $d_t=0 \iff (1-r^{9t+13})(1-r^{9t}) = r^{9t}(1+r)(r^3-1)(1+r^2)$.
Now we check the validity of this equation for $1\le t\le N-2=12$:
\begin{align*}
t&=1: & (1+r)(1+r^{2}) &\neq r^{9}(1+r)(r^3-1)(1+r^2) ;
\\
t&=2: & (1-r^{3})(1-r^{4}) &\neq r^{4}(1+r)(r^3-1)(1+r^2);
\\
t&=3: & (1+r^5)(1+r^6) &\neq r^{13}(1+r)(r^3-1)(1+r^2) ;
\\
t&=4: & 2(1+r) &\neq r^{8}(1+r)(r^3-1)(1+r^2) ;
\\
t&=5: & (1-r^{2})(1-r^{3}) &\neq r^{3}(1+r)(r^3-1)(1+r^2);
\\
t&=6: & (1+r^4)(1+r^6) &\neq r^{12}(1+r)(r^3-1)(1+r^2);
\\
t&=7: & 2(1-r^{6}) &\neq -(1+r)(r^3-1)(1+r^2);
\\
t&=8: & (1-r)(1-r^{2}) &\neq r^{2}(1+r)(r^3-1)(1+r^2);
\\
t&=9: & (1+r^3)(1+r^4) &\neq r^{11}(1+r)(r^3-1)(1+r^2);
\\
t&=10: & (1-r^{5})(1-r^{6}) &\neq r^{6}(1+r)(r^3-1)(1+r^2);
\\
t&=11: & 0 &\neq r(1+r)(r^3-1)(1+r^2);
\\
t&=12: & (1+r^2)(1+r^3) &\neq r^{10}(1+r)(r^3-1)(1+r^2).
\end{align*}
Thus $d_t\neq 0$ for all $1\le t\le 12$ so $y_1^{14} \neq 0$ by Proposition \ref{prop:dt-yn}, and hence $14\beta_1$ is a root by Lemma \ref{le:yNneq0isroot}. Therefore $\GK \toba(V)=\infty$ by Proposition \ref{pro:rootweight1}.

Assume that $q_{11}=\eta\in\G_4'$, $q_{22}^2\qt_{12}=1$: to simplify the notation, set $q=q_{22}$, so $\qt_{12}=q^{-2}$. Then $-1=q_{11}\qt_{12}^{\, 2}q_{22}=\eta q^{-3}$, so $q^3=-\eta$. The matrix $(t_{ij})_{i,j\in\I}$ of $\cR^1(V)$ has diagram
$\xymatrix{\overset{\eta}{\circ} \ar@{-}[r]^{-q^{2}} & \overset{-q^{-2}}{\circ}}$. Thus $a_{21}^{\cR^1(V)}=-1$ and then $\cR^1(V)$ has finite root system, since $\GK \toba(\cR^1(V))=\GK \toba(V)<\infty$.

Assume now $q_{11}^3\qt_{12}=1$, $q_{22}=\zeta\in\G_3'$. Let $q=q_{11}$ so $\qt_{12}=q^{-3}$. The matrix $(t_{ij})_{i,j\in\I}$ of $\cR^2(V)$ has diagram
$\xymatrix{\overset{-1}{\circ} \ar@{-}[r]^{\zeta^2 q^{3}} & \overset{\zeta}{\circ}}$. Thus $a_{12}^{\cR^1(V)}=-1$ and then $\cR^2(V)$ has finite root system, since $\GK \toba(\cR^2(V))=\GK \toba(V)<\infty$.

If $q_{11}=\eta\in\G_4'$, $q_{22}=\zeta\in\G_3'$, then $-1=q_{11}\qt_{12}^{\, 2}q_{22}=\eta q^2\zeta$, where $q=\qt_{12}$. Hence the matrix $(t_{ij})_{i,j\in\I}$ of $\cR^2(V)$ has diagram 
$\xymatrix{\overset{-1}{\circ} \ar@{-}[r]^{q^{-1}\zeta^2} & \overset{\zeta}{\circ}}$. Thus 
$\cR^2(V)$ has finite root system, and $V$ too.


\begin{thebibliography}{AHS}
\bibitem[A]{A-leyva} N. Andruskiewitsch. \emph{An Introduction to Nichols Algebras}. In Quantization, Geometry and Noncommutative Structures in Mathematics and Physics. 
A. Cardona, P. Morales, H. Ocampo, S. Paycha, A. Reyes, eds., Springer (2017), 135--195.

\bibitem[AA1]{AA}
N. Andruskiewitsch, I. Angiono. \emph{On Nichols algebras with generic braiding},
in  \emph{Modules and
Comodules}, T. Brzezinski; J. L. G\'omez Pardo;
I. Shestakov; P. F. Smith (Eds.). Trends in Mathematics (2008), 47--64.


\bibitem[AA2]{AA-diag-survey} N. Andruskiewitsch, I. Angiono. \emph{On Finite dimensional Nichols algebras of diagonal type}. 
Bull. Math. Sci. \textbf{7} (2017), 353--573. 

\bibitem[AAH]{AAH} N. Andruskiewitsch,  I. Angiono,   I. Heckenberger.
\emph{On finite GK-dimensional Nichols algebras over abelian groups}. \texttt{\small arXiv:1606.02521}.




\bibitem[AS1]{AS Pointed HA} N. Andruskiewitsch, H.-J. Schneider. \emph{Pointed Hopf algebras}, New directions in Hopf algebras,
MSRI series Cambridge Univ. Press (2002), 1--68 .


\bibitem[An1]{A-jems} I. Angiono. {\em A presentation by generators and relations of Nichols algebras
of diagonal type and convex orders on root systems},  J. Europ. Math. Soc.  \textbf{17} (2015), 2643--2671.


\bibitem[C+]{carbone}
L. Carbone, S. Chung, C. Cobbs, R. McRae, D. Nandi, Y. Naqvi, D. Penta.
\emph{Classification of hyperbolic Dynkin diagrams, root lengths and Weyl group orbits}.
J. Phys. A: Math. Theor. \textbf{43}  (2010) (15): 155209.

\bibitem[CH]{CH} M. Cuntz, I. Heckenberger. \emph{Weyl groupoids with at most three objects}. J. Pure Appl. Algebra \textbf{213} (2009), 1112--1128.





\bibitem[GH]{GH} M. Gra\~na, I. Heckenberger. \emph{On a factorization of graded Hopf algebras using Lyndon words}. J. Algebra \textbf{314} (2007), 324--343.


\bibitem[H1]{H-inv} I. Heckenberger. {\em The Weyl groupoid of a Nichols algebra
of diagonal type}, Inventiones Math. \textbf{164} (2006), 175--188.

\bibitem[H2]{H-classif}  I. Heckenberger.  {\em Classification of arithmetic
root systems}, Adv. Math. \textbf{220} (2009), 59--124.

\bibitem[HS]{HS-RsWg}  I. Heckenberger, H.-J. Schneider. {\em Root systems and Weyl
groupoids for Nichols algebras}, Proc. London Math. Soc. \textbf{101} (2010), 623--654.

\bibitem[Kh]{Kh} V. Kharchenko. \emph{A quantum analog of the Poincare-Birkhoff-Witt theorem}, Algebra and Logic \textbf{38} (1999), 259--276.



\bibitem[K]{K} V. Kac.
\emph{Infinite-dimensional Lie algebras}. Third edition. Cambridge University Press, Cambridge, 1990. xxii+400 pp.

\bibitem[KL]{KL} G. Krause, T. Lenagan,.\emph{Growth of algebras and Gelfand-Kirillov dimension}. Revised edition.
Graduate Studies in Mathematics, \textbf{22}. American Mathematical Society, Providence, RI, 2000. x+212 pp

\bibitem[Mo]{Mo} S. Montgomery. \emph{Hopf algebras and their action on rings},
CBMS Regional Conference Series \textbf{82} (1993).

\bibitem[R]{R quantum groups} M. Rosso. \emph{Quantum groups and quantum shuffles}. Invent. Math. \textbf{133} (1998), 399--416.

\bibitem[Y]{Y} H. Yamane. \emph{Representations of a $\Z/3\Z$-quantum group}.
Publ. Res. Inst. Math. Sci.  \textbf{43} (2007),  75--93.

\end{thebibliography}
\end{document}